\documentclass[12pt]{amsart} 
\usepackage[T1]{fontenc}
\usepackage{latexsym,amsmath,amscd,amssymb,lmodern}
\usepackage{lscape}
\usepackage{mathdots}
\usepackage{array}
\usepackage{pdfsync}
\usepackage[all]{xy}
\usepackage[english]{babel}

\usepackage{amsthm} 
\usepackage{xfrac}  
\usepackage{multirow} 
\usepackage{setspace}

\usepackage{hyperref}
\hypersetup{colorlinks=true,
	linkcolor=black,          
	citecolor=black,        
	filecolor=black,      
	urlcolor=cyan           
}

\def\addsymbol #1: #2#3{$#1$ \> \parbox{5in}{#2 \dotfill \pageref{#3}}\\}

\newcommand{\st}{\;|\;}
\newcommand{\xra}{\xrightarrow}

\newcommand{\ts}{\times}

\newcommand{\ot}{\otimes}

\renewcommand{\H}{\textrm{H}}

\newcommand{\Z}{{\mathbb{Z}}}
\newcommand{\Q}{{\mathbb{Q}}}
\newcommand{\R}{{\mathbb{R}}}
\newcommand{\C}{{\mathbb{C}}}
\newcommand{\Oc}{{\mathbb{O}}}

\newcommand{\cM}{\mathcal{M}}

\newcommand{\cR}{\mathcal{R}}

\newcommand{\chS}{\check{S}}

\newcommand{\fg}{\mathfrak{g}}
\newcommand{\fh}{\mathfrak{h}}

\newcommand{\fl}{\mathfrak{l}}
\newcommand{\fm}{\mathfrak{m}}

\newcommand{\fsl}{\mathfrak{sl}}

\newcommand{\liea}{\mathfrak{a}}
\newcommand{\liem}{\mathfrak{m}}
\newcommand{\lien}{\mathfrak{n}}
\newcommand{\liep}{\mathfrak{p}}
\newcommand{\lieq}{\mathfrak{q}}
\newcommand{\liet}{\mathfrak{t}}

\newcommand{\liemc}{\mathfrak{m}^{\mathbb{C}}}
\newcommand{\lieh}{\mathfrak{h}}
\newcommand{\liehc}{\mathfrak{h}^{\mathbb{C}}}
\newcommand{\lieg}{\mathfrak{g}}
\newcommand{\liel}{\mathfrak{l}}
\newcommand{\liez}{\mathfrak{z}}

\newcommand{\al}{\alpha}
\newcommand{\be}{\beta}
\newcommand{\de}{\delta}
\newcommand{\ga}{\gamma}
\newcommand{\lam}{\lambda}

\newcommand{\De}{\Delta}
\newcommand{\Ga}{\Gamma}

\newcommand{\la}{\langle}
\newcommand{\ra}{\rangle}

\newcommand{\PU}{\mathrm{PU}}

\newcommand{\PSL}{\mathrm{PSL}}
\newcommand{\PSO}{\mathrm{PSO}}
\newcommand{\PSU}{\mathrm{PSU}}
\newcommand{\SU}{\mathrm{SU}}
\newcommand{\U}{\mathrm{U}}
\newcommand{\GL}{\mathrm{GL}}
\newcommand{\SL}{\mathrm{SL}}
\newcommand{\SO}{\mathrm{SO}}
\newcommand{\Sp}{\mathrm{Sp}}
\newcommand{\SSS}{\mathrm{S}}
\newcommand{\Spin}{\mathrm{Spin}}
\newcommand{\E}{\mathrm{E}}
\newcommand{\F}{\mathrm{F}}

\newcommand{\OO}{\mathrm{O}}

\DeclareMathOperator{\ad}{ad}
\DeclareMathOperator{\Ad}{Ad}

\DeclareMathOperator{\rk}{rk}

\DeclareMathOperator{\reg}{reg}
\DeclareMathOperator{\im}{im}

\DeclareMathOperator{\Hom}{Hom}

\DeclareMathOperator{\End}{End}

\DeclareMathOperator{\Id}{Id}

\DeclareMathOperator{\Herm}{Herm}

\DeclareMathOperator{\Aut}{Aut}
\DeclareMathOperator{\Out}{Out}
\DeclareMathOperator{\Conj}{Conj}
\DeclareMathOperator{\Isom}{Isom}
\DeclareMathOperator{\Mat}{Mat}

\DeclareMathOperator{\Sym}{Sym}
\DeclareMathOperator{\Skew}{Skew}

\DeclareMathOperator{\Tr}{Tr}
\DeclareMathOperator{\Int}{Int}

\newcommand{\lra}{\longrightarrow}

\newtheorem{theorem}{Theorem}[section]
\newtheorem{corollary}[theorem]{Corollary}
\newtheorem{lemma}[theorem]{Lemma}
\newtheorem{proposition}[theorem]{Proposition}

\newtheorem{definition}[theorem]{Definition}
\newtheorem{remark}[theorem]{Remark}

\title[Higgs bundles, the Toledo invariant and the Cayley correspondence]
{Higgs bundles, the Toledo invariant\\ and the Cayley correspondence}

\author[Olivier Biquard]{Olivier Biquard}
 \address{Universit\'e Pierre et Marie Curie et \'Ecole Normale Sup\'erieure, UMR 8553 du CNRS} 
 \email{olivier.biquard@ens.fr}

 \author[Oscar Garc{\'\i}a-Prada]{Oscar Garc{\'\i}a-Prada}
 \address{Instituto de Ciencias Matem\'aticas \\
 	 CSIC-UAM-UC3M-UCM \\ Nicol\'as Cabrera, 13--15 \\ 28049 Madrid \\ Spain}
 \email{oscar.garcia-prada@icmat.es}

 \author[Roberto Rubio]{Roberto Rubio}
 \address{IMPA\\
 Estrada Dona Castorina 110\\
 Rio de Janeiro,  22460-320 \\
 Brasil}
\curraddr{Weizmann Institute of Science\\
234 Herzl St, Rehovot, 7610001\\ Israel}
\email{roberto.rubio@weizmann.ac.il}

\thanks{
  The second author is partially supported by the Spanish MINECO under 
  the ICMAT Severo Ochoa grant No. SEV-2011-0087, and under grant 
  No. MTM2013-43963-P.  The third author was supported by a predoctoral I3P-JAE grant from CSIC, a scholarship of the Ayuntamiento de Madrid in the Residencia de Estudiantes, and a grant from the project Interactions of Low-Dimensional Topology and Geometry with Mathematical Physics (European Science Foundation).
} 

 \subjclass[2000]{Primary 14H60; Secondary 57R57, 58D29} 

\begin{document}

\maketitle

\begin{abstract}
  Motivated by the study of the topology of the  character 
  variety for a non-compact Lie group of Hermitian type $G$,  
  we undertake a uniform
  approach, independent of classification theory of Lie groups, to
  the study of the moduli space of $G$-Higgs bundles over a compact Riemann 
  surface. We give an intrinsic definition of the  Toledo invariant of a 
  $G$-Higgs bundle which relies on the Jordan algebra structure of the isotropy 
  representation for groups defining a symmetric space of tube type, and  
  prove a general Milnor--Wood type bound of this invariant 
  when the $G$-Higgs bundle is semistable.
  Finally, we prove rigidity results when the Toledo invariant is maximal, 
  establishing in particular a Cayley correspondence when $G$ is of tube 
  type,  which reveals new topological invariants only seen in particular 
  cases from the character variety viewpoint.
\end{abstract}

\section{Introduction}\label{sec:intro}

Non-abelian Hodge theory establishes a homeomorphism between the character
variety or moduli of representations of the fundamental group of  a 
compact Riemann surface $X$ in a real non-compact reductive Lie group $G$ and
the moduli space of $G$-Higgs bundles over $X$. One of the most successful 
applications of this correspondence is to the study of the topology
of the character variety by means of Morse theory and other localization 
methods on the moduli space of Higgs bundles, taking advantage of the 
fact that the moduli space of $G$-Higgs bundles is a complex algebraic variety.
In this paper, we look at groups
of Hermitian type from the Higgs bundles viewpoint and show that the algebraic structure of the isotropy representation can be used to give a simple and intrinsic definition of the Toledo invariant, and is responsible for a Milnor--Wood type inequality and rigidity phenomena in the moduli space. Our results provide the starting point for a systematic
general study of the topology of the moduli spaces, of which very little
is  known besides some specific examples.

To briefly explain the basics of Higgs bundle theory over a
compact Riemann surface $X$ of genus $g\geq 2$, let $G$ be 
a real reductive Lie group and $H\subset G$ a maximal compact
subgroup. Fixing  an invariant metric on the Lie algebra $\lieg$ of $G$,
we have an orthogonal decomposition $\lieg=\lieh +\liem$, where $\lieh$
is the Lie algebra of $H$. From the isotropy representation $H\to \Aut(\liem)$
we obtain the representation $\Ad:H^\C\to \Aut(\liemc)$. A $G$-Higgs bundle on $X$ is a pair $(E,\varphi)$
consisting of a holomorphic principal $H^\C$-bundle $E$ and a holomorphic
section $\varphi$ (the Higgs field) of the bundle $E(\liem^\C)\otimes K$, 
where $E(\liem^\C)$ is the  $\liemc$-bundle associated to $E$ via the 
representation $\Ad$, and $K$ is  the canonical line bundle of $X$.
We will also consider $L$-twisted $G$-Higgs bundles, replacing  $K$ by an
arbitrary line bundle $L$ over $X$. 
There are natural notions of stability, semistability, and polystability for 
these objects, leading to  corresponding
moduli spaces (see  \cite{GGM09}). 


In this paper we study the case of a connected non-compact real simple Lie group $G$ of Hermitian type with finite centre. 
In this situation 
the centre $\liez$ of $\lieh$ is isomorphic to $\R$, and  
the adjoint action of a special element $J\in \liez$ defines an almost complex 
structure on $\liem=T_o(G/H)$, where $o\in G/H$ corresponds to the coset $H$, 
making the  symmetric space $G/H$ into a K\"ahler manifold. 
The almost complex structure $\ad(J)$ gives a decomposition 
$\liem^\C=\liem^+ + \liem^-$ in $\pm i$-eigenspaces, which is $H^\C$-invariant.
An immediate consequence of  this decomposition  for a $G$-Higgs bundle 
$(E,\varphi)$ is that it  gives a 
bundle decomposition $E(\liem^\C)=E(\liem^+) \oplus E(\liem^-)$ 
and hence the Higgs field decomposes as $\varphi=(\varphi^+,\varphi^-)$,
where $\varphi^+\in H^0(X,E(\liem^+)\otimes K)$ and 
$\varphi^-\in H^0(X,E(\liem^-)\otimes K)$. Groups of Hermitian type fall into two classes: tube and non-tube type, depending whether their Harish-Chandra realization as a bounded domain is biholomorphic or not to a tube-type domain (see \cite{KW65}). The isotropy representation of a tube-type group naturally carries a Jordan algebra structure whose determinant is semi-equivariant by the action of the group $H^\C$. This semi-equivariance is described by a character of $\fh^\C$, which is the base for 
our introduction (Definition \ref{def:Toledo-character}) of the Toledo character $\chi_T:\lieh^\C \to \C$:
$$\chi_T(Y) = \frac1N \langle-i J, Y\rangle,$$
where $\langle\cdot,\cdot\rangle$ is the Killing form on $\lieg$, and $N$ is the dual Coxeter number. Moreover, maybe up to multiplication by an integer, $\chi_T$ lifts to a character $\tilde{\chi}_T$ of $H^\C$.

Our intrinsic new definition of the Toledo invariant (Definition \ref{def:Toledo-invariant}) is 
$$ \tau=\tau(E):=\deg(E(\tilde{\chi}_T)).$$ 
This still makes sense if only an integral multiple $\chi_T$ lifts to $H^\C$.

The Toledo invariant $\tau$ is a topological invariant attached to a $G$-Higgs bundle $(E,\varphi)$ which is key to the study of the moduli space. Another very important feature of the Hermitian condition is that the stability criterion depends on an element $\alpha\in i\liez$, hence basically a real number. We then define the moduli space $\cM^\alpha(G)$ of $\alpha$-polystable $G$-Higgs bundles over $X$. 

The case when $\alpha=0$, which will be denoted by $\cM(G)$ and referred to as polystable bundles, is of special significance as $\cM(G)$  is homeomorphic, by non-abelian Hodge theory (see \cite{GGM09}),  to the moduli space $\cR(G)$  of reductive representations of the fundamental group of $X$ in $G$. Although the moduli spaces $\cM^\alpha(G)$ for  $\alpha\neq 0 $ are not a priori related with representations of the fundamental group, they turn out to play an important role in the study of the topology of $\cM(G)$  and hence 
$\cR(G)$. This is  a powerful motivation for us to consider the study of $\al$-semistable $G$-Higgs bundles and prove one of the main results of this
paper (Theorem \ref{theo:ineq-rk}).
\begin{theorem}\label{theo:ineq-rk-intro} Let $\al\in i\liez$ such that $\al=i\lambda
  J$ for $\lambda\in\R$. Let $(E,\varphi^+,\varphi^-)$ be an $\al$-semistable $G$-Higgs
  bundle. Then, the Toledo invariant of $E$ satisfies:
  $$-\rk(\varphi^+)(2g-2)-\left(\frac{\dim \liem}{N}-\rk(\varphi^+)\right)\lambda \leq
  \tau \leq \rk(\varphi^-)(2g-2)-\left(\frac{\dim \liem}{N}-\rk(\varphi^-)\right)\lambda,$$ where $N$
  is the dual Coxeter number. In the tube case, this simplifies to:
  $$-\rk(\varphi^+)(2g-2)-(r-\rk(\varphi^+))\lambda \leq \tau \leq \rk(\varphi^-)(2g-2)-(r-\rk(\varphi^-))\lambda.$$
\end{theorem} 
The ranks $\rk(\varphi^+)$ and $\rk(\varphi^-)$ can be defined in the tube case using the Jordan  algebra structure, and reducing the non-tube case to  the tube situation, by means of the maximal subspace of tube type that always exists. In many of the classical cases, the spaces $\liem^\pm$ are spaces of matrices and these ranks coincide with the familiar rank of a matrix. The maximum value of these ranks is given by the rank of the symmetric space $\rk(G/H)$. For $\alpha=0$ one obtains as a consequence the Milnor-Wood inequality for semistable $G$-Higgs bundles,
$$
|\tau|\leq \rk(G/H)(2g-2),
$$
proved for representations in \cite{BIW10}. This inequality is therefore being extended in two ways: finding more accurate bounds and considering $\alpha$-semistability for a parameter $\alpha$.

We then focus on the study of 
$G$-Higgs bundles for which the Toledo invariant attains the bound 
in the Milnor--Wood inequality,
that is, $\tau=\pm  \rk(G/H)(2g-2)$. We call these, by analogy with the terminology applied to surface group representations, maximal $G$-Higgs bundles. 
Maximal representations --- and hence maximal Higgs bundles --- have special
significance in the context of `higher Teichm\"uller theory' 
since they provide examples of Anosov representations, and are 
related to geometric structures of various
kinds, in a similar way to that of Hitchin representations of the
fundamental group of the surface in a split real form 
(see e.g. \cite{goldman,burger-iozzi-labourie-wienhard:2005,hitchin92,labourie,guichard-wienhard:2008,BIW10,BGG13}). 

In our study, the tube-type condition plays a fundamental role. 
If $G$ is of tube type we construct
a bijective correspondence between maximal $G$-Higgs bundles and 
$K^2$-twisted $H^*$-Higgs bundles over $X$, where $H^*\subset H^\C$ is the
non-compact dual of $H$, as defined in Definition \ref{def:non-compact-dual}.  
Our main result is 
Theorem \ref{th:cayley-correspondence}.
\begin{theorem}[Cayley correspondence]\label{th:cayley-correspondence-intro}  
  Let $G$ be a connected non-compact  real simple Hermitian Lie group of tube type
  with finite centre. Let $H$ be a maximal compact subgroup of $G$ and  $H^*$ be 
  the non-compact dual of $H$ in $H^\C$. 
  Let $J$ be the element in $\liez$ (the centre of $\lieh$) defining
  the almost complex structure on $\liem$. 
  If the order of $e^{2\pi J}\in H^\C$ divides $(2g-2)$, then there is an 
  isomorphism of complex algebraic varieties
  $$
  \cM_{\max} (G) \cong  \cM_{K^2}(H^*).
  $$
\end{theorem}
It is useful to observe that the hypothesis on $J$ is always satisfied for the adjoint group.

One of the immediate consequences of Theorem 
\ref{th:cayley-correspondence} is the existence of other invariants
attached to a maximal $G$-Higgs bundle in the tube case. These are the 
topological  invariants of the  corresponding Cayley partner. 
These `hidden' invariants are not apparent from the 
point of view of the corresponding maximal representation and, 
as it has been seen for classical groups,
play a crucial
role in the computation of connected components of 
$\cM_{\max}(G)$. 

Maximal Higgs bundles in the non-tube case present also very interesting rigidity phenomena. Our main result in this case is Theorem \ref{theo:non-tube-rigidity}, where $\lieg_T$ is the maximal tube subalgebra of $\lieg$.
\begin{theorem}\label{theo:non-tube-rigidity-intro}
  Let $G$ be a simple Hermitian group of non-tube type and let $H$ be its
  maximal  compact subgroup. Then, there are no stable $G$-Higgs bundles with 
  maximal Toledo invariant. In fact, every polystable maximal $G$-Higgs bundle 
  reduces to a stable $N_G(\lieg_T)_0$-Higgs bundle, where $N_G(\lieg_T)_0$ is 
  the identity component of the normalizer of $\lieg_T$ in $G$.
\end{theorem}
In particular, the dimension of the moduli space of maximal
$G$-Higgs bundles is smaller than expected --- this rigidity phenomenon
is  very rare in the context of  surface groups, and is more frequent for 
representations of the fundamental group of higher-dimensional K\"ahler 
manifolds. This theorem implies that the moduli space fibers
over the moduli space of maximal $G_T^{\Ad}$-Higgs bundles, where $G_T^{\Ad}$
is the adjoint group of the maximal subgroup  of tube type $G_T\subset G$,
where the  fibre is a connected moduli space of bundles for a certain reductive 
complex Lie group (see  Theorem \ref{fibration}). In particular this allows us
to obtain results on the connectedness of $\cM_{\max}(G)$.

A brief description of the sections of the paper is as follows. In Section
\ref{sec:GHTandRT} we review some basic facts about groups of Hermitian type,
the Cayley transform and define the Toledo character. In Section 
\ref{chap:higgs-bundles} we introduce the basics of non-abelian Hodge 
theory relating Higgs bundles over a compact Riemann surface 
to representations of the fundamental group of the surface.
In Section \ref{sec:al-milnor-wood-inequality} we initiate the study of 
$G$-Higgs bundles for a group $G$ of Hermitian type. We define the Toledo
invariant of a $G$-Higgs bundle and prove the Milnor--Wood inequality. We also 
study involutions defining isomorphisms between moduli space of $G$-Higgs
bundles with opposite Toledo invariant. In Section 
\ref{chap:cayley-correspondence} we consider maximal $G$-Higgs bundles
when $G$ is of tube type and establish the Cayley correspondence. 
Finally, in Section \ref{chap:non-tube-domains} we study maximal $G$-Higgs
bundles when $G$ is not of tube type and prove the rigidity phenomena taking place.

As mentioned above, our results should provide the starting point for an intrinsic study of the 
topology and geometry of the moduli spaces of $G$-Higgs bundles when $G$ is a
group of Hermitian type, in particular for the counting of connected 
components of the moduli 
space. This had been carried out to some extent for some of the classical
groups on a case by case basis (\cite{hitchin87,gothen,GM04,BGG03,BGG06,BGG15,GGM13}), making use of the classification theorem of Lie groups,  
but no general principle emerged. Our present intrinsic approach offers a new understanding of the Toledo invariant, Milnor--Wood bound and rigidity phenomena, which are fundamental to the topological study of the moduli space. A preliminary version of some of our results is in \cite{Rub12}.

\noindent {\bf Acknowledgements.}
We wish to thank Nigel Hitchin for very useful suggestions.
The second and third authors wish to thank the \'Ecole Normale Sup\'erieure (Paris) for hospitality and support.

\section{Groups of Hermitian type and the Toledo character} 
\label{sec:GHTandRT}

The results surveyed in the first part of this section can be found in Chapter
VIII of  \cite{helgason} and Part III of \cite{FKKLR00}.

\subsection{Hermitian symmetric spaces and Cayley transform}
\label{cayley-transform}

Let $G/H$ be an irreducible Hermitian symmetric space of non-compact
type, where $G$ is a connected, non-compact real simple Lie group of Hermitian type
with finite centre. Such a group is characterized by the fact
that the centre $Z(H)$ of a maximal compact subgroup $H$ is isomorphic
to $\U(1)$. Note that the same symmetric space is obtained by starting with the adjoint group of $G$, which acts effectively on $G/H$, or any of its finite coverings.

We denote by $\lieg=\lieh+\liem$ the corresponding Cartan decomposition and by $\theta$ the Cartan
involution, so we have $[\lieh,\lieh]\subset \lieh$ and $[\lieh,\liem]\subset \liem$, which
lifts to the isotropy representation $\Ad:H\to\Aut(\liem)$.
Let $H^\C$, $\liehc$ and $\liemc$ be the complexifications of $H$,
$\lieh$ and $\liem$ respectively.  The almost complex structure $J_0$
on $\liem=T_o(G/H)$, where $o\in G/H$ corresponds to the coset $H$, 
is induced by the adjoint action of an element $J\in
\liez(\lieh)$, so $J_0=\ad(J)|_{\liem}$. Since $J_0^2=-\Id$, we decompose
$\liem^\C$ into $\pm i$-eigenspaces for $J_0$: $\liem^\C=\liem^+
+ \liem^-$. Both $\liem^+$ and $\liem^-$ are abelian, $[\lieh^\C,\liem^\pm]\subset
\liem^\pm$, and there are $\Ad(H)$-equivariant isomorphisms $\liem\cong \liem^\pm$ given by $X\mapsto \frac{1}{2}(X\mp iJ_0X)$.

Consider a maximal abelian subalgebra $\liet$ of $\lieh$. Its
complexification $\liet^\C$ gives a Cartan subalgebra of $\lieg^\C$, for
which we consider the root system $\Delta=\Delta(\lieg^\C,\liet^\C)$ and the
decomposition $\lieg^\C=\liet^\C+\sum_{\al\in\Delta} \lieg^\C_\al$. Since
$\ad(\liet^\C)$ preserves $\lieh^\C$ and $\liem^\C$, $\lieg_\al^\C$ must lie
either in $\lieh^\C$ or in $\liem^\C$. If $\lieg_\al^\C\subset \lieh^\C$ (resp.
$\lieg_\al^\C\subset \liem^\C$) we say that the root $\al$ is \textbf{compact}
(resp. \textbf{non-compact}) and denote the set of such roots by $\Delta_C$
(resp. $\Delta_Q$).  We choose an ordering of the roots in
such a way that $\liem^+$ (resp. $\liem^-$) is spanned by the root vectors
corresponding to the non-compact positive (resp. negative) roots. We
use the superscript $+$ (resp. $-$) to denote the positive
(resp. negative) roots from a set of roots: $\Delta^+$, $\Delta_C^+$, $\Delta_Q^+$
(resp. $\Delta^-$, $\Delta_C^-$, $\Delta_Q^-$). Then,
$$\liem^{\pm}=\sum_{\al\in\Delta^\pm_Q} \lieg^\C_\al .$$

We denote by $\langle\cdot,\cdot\rangle$ an invariant form on $\lieg^\C$, a constant
multiple of the Killing form (most often, the Killing form itself).
For each root $\al\in\Delta$, let $H_\al\in i\liet$ be the dual of $\al$, 
i.e., $$\al(Y)=\langle Y,H_\al\rangle\qquad \textrm{for } Y\in
i\liet.$$ Define, as usual, $h_\al= \frac{2H_\al}{\langle H_\al,H_\al\rangle} \in
i\liet$, and $e_\al\in\lieg^\C_\al$ such that $[e_\al,e_{-\al}]=h_\al$ and
$\tau e_\al=-e_{-\al}$, where $\tau$ is the involution of $\lieg^\C$ fixing
its compact real form $\lieh+i\liem$.  We define a real basis of $\liem$ by
taking for each $\alpha\in \Delta^+_Q$ the basis $(x_\al=e_\al + e_{-\al},y_\al
=i(e_\al-e_{-\al}))$ of $\lieg^\C_\alpha\oplus\lieg^\C_{-\alpha}$.

Two roots $\al,\be\in\De$ are said to be \textbf{}{strongly orthogonal}
if neither $\al+\be$ nor $\al-\be$ is a root (equivalently
$[\lieg^{\al},\lieg^{\pm\be}]=\{0\}$). A \textbf{system of strongly orthogonal
  roots} is a maximal set of strongly orthogonal positive non-compact roots.  It
has a number of elements equal to the rank $r=\textrm{rk}(G/H)$ of the
symmetric space $G/H$, i.e., the maximal dimension of a flat, totally
geodesic submanifold of $G/H$. Moreover, for
two strongly orthogonal roots $\ga\neq \ga'$ we have 
\begin{equation}\label{eq:relations-e-h-Gamma}
  [e_{\pm\ga},e_{\pm\ga'}] = 0, \quad [e_{\pm \ga}, h_{\ga'}] = 0.
\end{equation}

For a strongly orthogonal system of roots $\Ga$, consider
\begin{equation*}
  x_\Gamma=\sum_{\ga\in\Gamma} x_\ga, \quad
  y_\Gamma=\sum_{\ga\in\Gamma} y_\ga, \quad
  e_\Gamma=\sum_{\ga\in\Gamma} e_\ga, \quad
  c=\exp\left(\frac{\pi}{4} i y_\Gamma\right) \in U\subset G^\C, 
\end{equation*}
where $G^\C$ is the simply connected Lie group with Lie algebra $\lieg^\C$ and
$U$ is its compact real form  (with Lie algebra $\lieh\oplus i\liem$).  We define the \textbf{Cayley transform} as
the action of the element $c$ on the Lie algebra $\lieg^\C$ by
$\Ad(c):\lieg^\C\to \lieg^\C$.

The Cayley transform $\Ad(c)$ satisfies $\Ad(c^8)=\Id$, $\Ad(c)\circ \theta=\theta\circ
\Ad(c^{-1})$ for the Cartan involution $\theta$, and consequently $\Ad(c^4)$ preserves $\lieh$ and $\liem$,
even though $\Ad(c)$ does not preserve $\lieg$. Since
$(\Ad(c^4))^2=\Id$, either $\Ad(c^4)=\Id$ (then the Hermitian
symmetric space is said of \textbf{tube type}), or (for
\textbf{non-tube type}) we can decompose $\lieh$ and $\liem$ into $\pm
1$-eigenspaces for $\Ad(c^4)$:
\begin{align*}
  \liem&=\liem_T+\liem_2&
  \lieh&=\widetilde{\lieh}_T+\lieq_2. 
\end{align*}
We define $\widetilde{\lieg}_T=\widetilde{\lieh}_T+\liem_T,$ which is a
Lie algebra as $\widetilde{\lieh}_T$ acts on $\liem_T$. Since $\widetilde{\lieh}_T$ may have a non-trivial ideal, we define $\lieh_T=[\liem_T,\liem_T]$ and $\lieg_T=\lieh_T+\liem_T$ to get the Cartan decomposition $\lieg_T=\lieh_T+\liem_T$, associated to an irreducible Hermitian symmetric space.
The subalgebras $\widetilde{\lieg}_T$ and $\widetilde{\lieh}_T$ are then
the normalizers $\widetilde{\lieg}_T=\lien_{\lieg}(\lieg_T)$ and
$\widetilde{\lieh}_T=\lien_{\lieh}(\lieh_T)$.
We also use the notation $\liem^\pm_T=\liem^\C_T\cap \liem^\pm$, $ \liem^\pm_2=\liem^\C_2\cap \liem^\pm$.

We denote by $G_T$ and $H_T$ the (connected) subgroups of $G$ with Lie
algebras $\lieg_T$ and $\lieh_T$. The group $H_T$ is a maximal compact
subgroup of $G_T$. The subgroups of $G$ with Lie algebras
$\widetilde{\lieg}_T$ and $\widetilde{\lieh}_T$ are $N_G(\lieg_T)_0$ and
$N_H(\lieh_T)_0$. The Cartan decomposition is
$\lien_{\lieg}(\lieg_T)=\lien_{\lieh}(\lieh_T)+\liem_T$ and the maximal compact
subgroup of $N_G(\lieg_T)_0$ is $N_H(\lieh_T)_0$, whose complexification
is $N_{H^\C}(\lieh_T^\C)_0$.

Since $\Ad(c^4)=\Id$ on $\lieg_T$, the Hermitian symmetric space
$G_T/H_T$ is now of tube type (the maximal `subtube' of $G/H$; this is
$G/H$ if it was already of tube type). Moreover, $\Ad(c^2)$ commutes
with $\theta$, so it preserves $\lieh_T$, and we get a decomposition into $\pm
1$-eigenspaces for $\Ad(c^2)$,
$$\lieh_T=\lieh'+i\liem'.$$

Related to this decomposition, there are two groups that will play a fundamental role in our work: the isotropy group $H'$ of $ie_\Gamma$ in $H_T$, whose Lie algebra is 
 $\lieh'$, and the non-compact dual of $H_T$, which we define as follows. 
\begin{definition}\label{def:non-compact-dual}
The subgroup $H^*_T\subset H^\C_T$ integrating the subalgebra $\lieh'+\liem'\subset \lieh_T^\C$ is called the non-compact dual of $H_T$. When $G/H$ is of tube type, we say that $H^*$ is the \textbf{non-compact dual of $H$}.
\end{definition}
The group $H'$ is thus the maximal compact subgroup of $H^*_T$ and $H_T^*/H'$ is a symmetric space of non-compact type (the non-compact dual of the Shilov boundary of $G_T/H_T$, which is given by $H_T/H'$). 
If $G/H$ is of tube type the symmetric bounded domain corresponding to 
$G/H$ is biholomorphic to the `tube' $V+i\Omega\subset V^\C$, where 
$V=\Ad(c)\lieg\cap \liem^+$ and $\Omega=H^*/H'$ is a symmetric cone in $V$.
We refer to Table \ref{tab:HSS} for the list of tube and non-tube type Hermitian 
groups (up to covering, and quotient by a subgroup of the centre).

The following lemma will be important for our Cayley correspondence.
\begin{lemma}\label{lemma:cayley-iso}
  The maps $\ad(e_\Gamma):\liem_T^-
  \to \liem'^\C$ and $\ad(e_\Gamma):\liem'^\C \to \liem_T^+$ are $\Ad(H'^\C)$-equivariant isomorphisms.
\end{lemma}
\begin{proof}
  This lemma is well-known, but it is useful to give a short
  proof. The map is clearly $\Ad(H'^\C)$-equivariant, so there remains
  to prove the isomorphism statement.  We can restrict ourselves to the tube
  type and forget the index T. Therefore $\Ad(c^4)=1$ and there is a
  decomposition $\liem=\liem_1\oplus\liem_{-1}$ into $\pm1$ eigenspaces for
  $\Ad(c^2)$. As $\Ad(c^2)J=-J$, the spaces $\liem_{\pm1}$ are exchanged
  by $J$. Since $\Ad(c)y_\Gamma=y_\Gamma$, there are maps
  $$ \liem_{-1} \xrightarrow{\ad y_\Gamma} i\liem' \xrightarrow{\ad y_\Gamma} \liem_{-1}. $$
  Since $\ker\ad y_\Gamma\subset \ker(\Ad(c)-1)$, both maps are injective, and
  therefore bijective. On the other hand $\ad(y_\Gamma)|_{\liem_1}=0$ because
  $\Ad(c^2)=\exp(i\frac \pi2 \ad y_\Gamma)=1$ on $\liem_1$, but the eigenvalues
  of $\ad(y_\Gamma)$ take only the values $\{\pm2, \pm1, 0\}$.

  Now complexify this picture: $\liem_{-1}^\C$ projects bijectively to
  $\liem_-$, so we obtain an isomorphism $\liem_{-1}^\C\simeq \liem^- \to
  \liem'^\C$ given for $X\in \liem_{-1}^\C$ (so $JX\in \liem_1^\C$) by
  $$ X+iJX \longmapsto \ad(y_\Gamma)X = \ad(y_\Gamma)(X+iJX) = \ad(ie_\Gamma)(X+iJX). $$
  The first isomorphism follows, the second one is similar.
\end{proof}

\subsection{Restricted root theory}
\label{appendix:restricted-root-theory}

Given a system $\Gamma=\{\gamma_1,\dots,\gamma_r\}$ of strongly orthogonal roots, let
$\liet^-=\sum_\Ga \R i h_\ga\subset \liet$, denote by $\pi:(\liet^\C)^*\to (i\liet^-)^*$
the restriction to $i\liet^-$. We will identify $\gamma_i$ with $\pi(\gamma_i)$. The restricted root theorem says that the image by $\pi$ of the root system $\Delta$ is
$$
\pi(\Delta)\cup \{0\} = 
\begin{cases}
  \{ \tfrac{\pm\gamma_i\pm\gamma_j}2, 1\leq i,j\leq r \} & \text{in the tube type case,} \\
  \{ \tfrac{\pm\gamma_i\pm\gamma_j}2, 1\leq i,j\leq r \} \cup \{ \pm\tfrac{\gamma_i}2, 1\leq i\leq r \} & \text{otherwise.}
\end{cases}
$$
Moreover, all the roots $\gamma_i$ have the same length, therefore we shall note
$$ \langle\gamma,\gamma\rangle := \langle\gamma_i,\gamma_i\rangle \quad \text{for any }i. $$

We divide the roots according to their projection: the compact positive roots can project to $0$, $-\frac12 \gamma_i$ or $\frac12(\gamma_j-\gamma_i)$ for $j>i$, so we set for $j>i$
\begin{align*}
  C_0 &= \Delta_C^+ \cap \pi^{-1}(0),& C_i &= \Delta_C^+\cap \pi^{-1}(-\tfrac12 \gamma_i),&
  C_{ij} &= \pi^{-1}(\tfrac12(\gamma_j-\gamma_i)).
  \intertext{Similarly the positive non-compact roots subdivide into the subsets (again $j>i$)}
  \Gamma&, & Q_i &= \Delta^+_Q\cap \pi^{-1}(\tfrac12 \gamma_i), &
  Q_{ij}&=\pi^{-1}(\tfrac12(\gamma_j+\gamma_i)).
\end{align*}

Actually, the following translations are bijections:
\begin{align}\label{eq:bijections-restricted-roots}
  C_i &\xra{+\ga_i} Q_i, & C_{ij} &\xra{+\ga_i} Q_{ij}.
\end{align}

The projected roots appear with certain multiplicities: $\pm\ga_j$ $(1\leq j\leq r)$ with multiplicity $1$, $\pm\frac{1}{2}\ga_j\pm\frac{1}{2}\ga_k$ ($j\neq k$) with multiplicity $a$, and possibly the roots $\pm\frac{1}{2}\ga_j$ with even multiplicity $b$. (In the tube case, $b=0$).

A relevant number in what follows is the dual Coxeter number $N$, an invariant of an irreducible root system, defined in general by considering the dual of the highest root $\theta$ expressed in the dual base of simple roots $\{\al_i\}$, $\theta^{\lor}=\sum a_i^{\lor} \al_i^{\lor}$: we then have $N=1+\sum a_i^{\lor}$. In our hermitian case, one has
\begin{equation}
  \label{eq:dual-coxeter-number}
  N = a(r-1)+b+2 = \frac1r (\dim_\C\liem^+ + \dim_\C\liem^+_T) ,
\end{equation}
where $\dim_\C\liem^+_T=\frac{r(r-1)}{2}a+r$ and
$\dim_\C\liem^+=\dim_\C\liem^+_T+rb$.  Finally:
\begin{lemma}\label{lem:N}
  For the Killing form, one has the equality $N = \frac1{\langle\gamma,\gamma\rangle}$.
\end{lemma}
\begin{proof}
  First observe the following useful formula: for $Y\in \lieh^\C$, one has
  \begin{equation}
    \langle J,Y\rangle= \Tr_{\liem^\C}(\ad(Y)\ad(J))
    = 2i\Tr_{\liem^+}(\ad Y).\label{eq:15}
  \end{equation}
  Therefore, for $\gamma\in \Gamma$,
  $$ \langle J,H_\gamma\rangle = 2i\Tr_{\liem^+}(\ad H_\gamma) = 2i \sum_{\alpha\in \Delta^+_Q} \langle\alpha,\gamma\rangle. $$
  Now we know all the orthogonal projections of roots in $\Delta^+_Q$ on $\Gamma$: we obtain each $\gamma_i$ once, $\frac12(\gamma_i+\gamma_j)$ with multiplicity $a$ and $\frac12\gamma_i$ with multiplicity $b$. Therefore
  $$ \langle J,H_\gamma\rangle = i \langle\gamma,\gamma\rangle ((r-1)a+b+2) = iN \langle\gamma,\gamma\rangle.$$
  On the other hand, $\langle J,H_\gamma\rangle=\gamma(J)=i$, so the result follows.
\end{proof}

\subsection{The Toledo character} \label{sec:Toledo-character}

We introduce the Toledo character associated to a simple Lie algebra $\lieg$ of Hermitian type as the character on the Lie algebra $\lieh^\C$ given as follows.

\begin{definition}\label{def:Toledo-character}
  The \textbf{Toledo character} $\chi_T:\lieh^\C \to \C$ is defined, for $Y\in \lieh^\C$, by
  $$\chi_T(Y) = \langle-i J, Y\rangle \langle\gamma,\gamma\rangle.$$
  This is independent of the choice of invariant form on $\lieg$. If one chooses the Killing form, then equivalently, from Lemma \ref{lem:N},
  $$ \chi_T(Y) = \frac1N \langle-i J, Y\rangle. $$
\end{definition}
Since $J$ is in the center, $\chi_T$ vanishes on $[\lieh^\C,\lieh^\C]$, hence determines a character.

We study now when the Toledo character lifts to a character of the
group $H^\C$. Note that this depends on the choice of the pair $(G,H)$
defining the same symmetric space. Let $Z_0^\C$ denote the identity
component of $Z(H^\C)$.
\begin{proposition}\label{prop:exponentiation-of-Toledo}
  Define $o_J$ to be the order of $e^{2\pi J}$ and $\ell=|Z_0^\C\cap
  [H^\C,H^\C]|$. For $q\in \Q$, the character $q \chi_T$ lifts to $H^\C$ if
  and only if $q$ is an integral multiple of
  \begin{equation}\label{eq:Toledo-character-group}
    q_T=\frac{\ell N}{o_J \dim \liem}.
  \end{equation}
\end{proposition}
\begin{proof}
  Observe that $\chi_T(J)=\frac1N \langle-iJ,J\rangle = \frac{\dim\liem}N$.
  From the isomorphism $\C^*\simeq Z_0^\C$ given by
  $e^\lambda\mapsto e^{-i\lambda o_JJ}$, it follows that $q\chi_T$ lifts to $Z_0^\C$ if and only if $q\frac{o_J \dim\liem}N \in \Z$.

  Now let $D=Z_0^\C\cap [H^\C,H^\C]$. Then $H^\C=[H^\C,H^\C]\times_DZ_0^\C$
  and every character of $H^\C$ comes from a character of
  $[H^\C,H^\C]\times Z_0^\C$ (and therefore of $Z_0^\C$) containing $D$ in
  its kernel. Since $D$ is a finite subgroup of $Z_0^\C\simeq \C^*$, it is
  cyclic, so it must be generated by $e^{2\pi\frac{o_J}\ell J}$. The
  proposition follows.  
\end{proof}

The value of $q_T$ in the standard examples is given in Table 
\ref{tab:exponent-tubetype}. 
Note in particular that $q_T=\frac12$ for
all classical groups except $\SO^*$, for which $q_T=1$. So for all
classical groups the Toledo character lifts to $H^\C$.
For the adjoint group, $o_J=1$ so $q_T=\frac{\ell N}{\dim\fm}$. In the tube case $N=\frac{\dim\fm}r$ so this gives $q_T=\frac \ell r$. The values of 
$q_T$ in the non-tube case are given in Table \ref{tab:exponent-adjoint}.

It follows from (\ref{eq:15}) that the lifted character
$\widetilde{\chi}_T$ can be interpreted as
\begin{equation}
  \widetilde{\chi}_T(h)=\det \Ad_{\liem^+}(h)^\frac{2}{N}.\label{eq:14}
\end{equation}
Of course, one has to take the power $q_T$ if only $q_T\chi_T$ is liftable.

Finally, the following lemma will prove later that the Toledo invariants defined from the two points of view of Higgs bundles and representations coincide.
\begin{lemma}\label{lemma:Toledo-Kahler-form}
  The Toledo character $\chi_T$ defines a $G$-invariant form on $G/H$ by 
  $$\omega(Y,Z)=i\chi_T([Y,Z]), \text{ for }Y,Z\in \liem.$$
  This form is the K\"ahler form of the unique $G$-invariant metric on $G/H$  with minimum holomorphic sectional curvature $-1$.
\end{lemma}
\begin{proof}
  Every invariant $p$-form $\omega$ on a symmetric space is closed
  \cite[p. 198]{helgason}, so the formula defines a closed homogeneous
  2-form on $G/H$. The associated metric reads
  $$ g(Y,Z) = \omega(Y,JZ) = \langle J,[Y,[J,Z]]\rangle\langle\gamma,\gamma\rangle = \langle Y,Z\rangle\langle\gamma,\gamma\rangle $$
  which is a positive metric since $\liem$ is non-compact. So the formula indeed defines a K\"ahler metric on $G/H$.
  
  Using the curvature $R(X,Y)=\ad_\liem [X,Y]$, we calculate the holomorphic sectional curvature for $X\in \liem$ such that $g(X,X)=1$: we have
  $$ \kappa(X) = g(R(X,JX)X,JX) = - \omega(R(X,JX)X,X) = -i\chi_T([[[X,JX],X],X]). $$
  For $X=x_\gamma$ we have $g(x_\gamma,x_\gamma)=\omega(x_\gamma,y_\gamma)=i\chi_T(-2ih_\gamma)=2\chi_T(h_\gamma)$ and
  $$ [[[x_\gamma,y_\gamma],x_\gamma],x_\gamma] = -8h_\gamma. $$
  In general, it is sufficient to do the calculations for $X=\sum_1^r\lambda_ix_{\gamma_i}$, so
  $$ \kappa(X) = - \frac{\sum_1^r \lambda_i^4}{(\sum_1^r \lambda_i^2)^2} \frac2{\chi_T(h_\gamma)}. $$
  But $\chi_T(h_\gamma)=\langle-iJ,h_\gamma\rangle\langle\gamma,\gamma\rangle=-2i\gamma(J)=2$, so the result follows.
\end{proof}

\subsection{Determinant and rank} \label{sec:jordan-algebra-structure}

We define now a determinant polynomial, $\det$, on $\liem^+_T$, whose degree
equals the rank of the symmetric space. This determinant is a familiar object
in Jordan algebra theory \cite{FK94}, but it can be introduced in an
elementary way as follows \cite[Lemma 2.3]{KV79}: it is the unique
$H'_0$-invariant 
polynomial on $\liem_T^+$ which restricts on $\liea^+=\oplus_1^r\C e_{\gamma_i}$ to
$$ \det \sum_1^r \lambda_ie_{\gamma_i} = \prod_1^r \lambda_i. $$
Here $H'_0$ is the identity component of $H'$.
The existence comes from the Chevalley theorem on invariant polynomials, since the Weyl group acts exactly by all permutations on the $(e_{\gamma_i})$ (see again \cite{KV79}).

The main useful property for us is the following equivariance:
\begin{lemma}\label{lemma:relation-det-char-to-q}
  Let $G$ be of tube type. For $h\in H^\C$ and $x\in \liem^+$ we have
  \begin{equation}\label{eq:8}
    \det(\Ad(h) x)=\widetilde{\chi}_T(h) \det(x),
  \end{equation}
  where $\widetilde{\chi}_T$ is the lifting of $\chi_T$ to $H^\C$.
\end{lemma}
Note that we implicitly assumed here that the lifting $\widetilde{\chi}_T$
exists, otherwise the same identity remains true after taking power
$q_T$. Again note that the lemma is basically known in Jordan algebra theory,
see \cite[Chapter VIII]{FK94} where the equivariance is established under the
action of $H^*$, the non-compact dual real form of $H$ in $H^\C$. The
extension  to $H^\C$ just requires the existence of the character 
$\tilde \chi_T$ on $H^\C$: this is not automatic because $H^*$ is 
non-compact, but of course was proved in the previous section.
\begin{proof}
  First, observe that if $\gamma_i$ is a strongly orthogonal root, then $[h_{\gamma_i},e_{\gamma_j}]=2\delta_{ij}e_{\gamma_j}$ and $\chi_T(h_{\gamma_i})=2$, so the requested equivariance property is true for the torus generated by the $h_{\gamma_i}$.

  Now observe that $\oplus_1^r \R h_{\gamma_i}\subset \liem'$ is a flat for the symmetric space $H/H'$ (we are in the tube type case). Therefore any element $X$ of $\liem'$ is conjugate under $H'_0$ to an element of $\oplus_1^r \R h_{\gamma_i}$, and it follows that (\ref{eq:8}) holds for all $\{\exp X, X\in \liem'\}$. Now $H'_0\exp \liem'$ is at least dense in $H$ (recall $H$ is connected, since $G$ is connected), so (\ref{eq:8}) is true for all $h\in H$. Of course, the property extends immediately to the complex group $H^\C$.
\end{proof}

We define a notion of rank on $\liem^+$. Choose $\liea^+=\oplus_1^r\C e_{\gamma_i}$. Any element of $\liem^+$ is conjugate under $H^\C$ to an element of $\liea^+$.
\begin{definition}\label{definition-rank}
  Let $x\in \liem^+$, and $y=\sum \lambda_ie_{\gamma_i}\in \liea^+$ be conjugate to $x$ under $H^\C$. Then we say that $x$ has rank $r'$ if $y$ has exactly $r'$ non-zero coefficients.
\end{definition}
This is well defined because the Weyl group acts only by permutations
on $\liea^+$. Also, in the tube case, one can give a more
intrinsic interpretation using the determinant: polarize the
determinant to get an $r$-linear map $C$ on $\liem^+$ such that
$C(x,\ldots,x)=\det(x)$; then the rank of $x$ is the maximal integer $r'$
such that the $(r-r')$-form $C(x,\ldots,x,\cdot,\ldots,\cdot)$ is not identically zero,
which is clearly an invariant notion.

\begin{remark}
  In the case of $\SU(p,q)$ the rank on $\liem^+$ specializes to the notion of rank for a rectangular matrix $q\ts p$. For $\Sp(2n,\R)$, the rank on $\liem^+$ is the rank for an element of $S^2(\C^n)$ seen as an endomorphism. 
\end{remark}

The following proposition plays an important role in what follows.
\begin{proposition}\label{prop:HC-trans}
  Let $1\leq r'\leq r$.  The group $H^\C$ acts transitively on the set of  elements of rank $r'$ in $\liem^+$. In particular, the set of regular (that is maximal rank) elements in $\liem^+$ is $H^\C/H'^\C$.
\end{proposition}
\begin{proof}
  It is sufficient to prove this on elements of $\liea^+$. Each $h_{\gamma_i}$ generates a one parameter subgroup acting only on $e_{\gamma_i}$ and fixing $e_{\gamma_j}$ for $j\neq i$. And the Weyl group acts transitively on the basis $(e_{\gamma_i})$, so the result follows.
\end{proof}

\section{Higgs bundles}
\label{chap:higgs-bundles}
\label{sec:g-higgs-defs}
\label{sec:moduli-spaces}
\label{sec:G-Higgs-herm}

In this section $G$ is a real reductive Lie group, not necessarily of
Hermitian type and not necessarily connected,
and $X$ is a compact  Riemann surface of
genus $g$. We fix a maximal compact 
subgroup $H$ of $G$. 
The Lie algebra $\lieg$ of $G$ is equipped with an involution $\theta$ 
that gives the Cartan decomposition  $\lieg=\lieh + \liem$,
where $\lieh$ is the Lie algebra of $H$. We fix a  metric  $B$ in $\lieg$
with respect to  which the Cartan decomposition is orthogonal.
This metric  is positive definite on $\liem$ and negative definite on $\lieh$.  
We have $[\liem,\liem]\subset \lieh$, $[\liem,\lieh]\subset\lieh$.
From the isotropy representation $H\to \Aut(\liem)$, we obtain the
representation $\Ad: H^\C \to \Aut(\liem^\C)$.
When $G$ is semisimple we take $B$ to be the Killing form.
In this case $B$ and a choice of a maximal compact subgroup $H$ 
determine a Cartan decomposition (see \cite{knapp} for details). 

\subsection{Basic definitions}

{\bf A $G$-Higgs bundle} on  $X$ consists of a holomorphic principal
$H^\C$-bundle  $E$ together with a holomorphic section $\varphi\in
H^0(X,E(\liem^\C)\ot K)$,
where  $E(\liem^\C)$ is the associated vector bundle with fibre 
$\liem^\C$ via the complexified isotropy representation, and $K$ is 
the canonical line bundle of $X$.

If $G$ is compact,  $H=G$ and $\liem=0$. A $G$-Higgs bundle is hence 
simply  a holomorphic principal  $G^\C$-bundle.
If  $G=H^\C$, where now $H$ is a compact Lie group,  $H$ is 
a maximal compact subgroup of $G$, and $\liem=i\lieh$. In this case, a $G$-Higgs 
bundle is a principal $H^\C$-bundle together with a section 
$\varphi\in H^0(X,E(\lieh^\C)\ot K)=H^0(X, E(\lieg) \ot K)$, where $ E(\lieg)$ is the
adjoint bundle. This is the original definition for complex Lie groups given 
by Hitchin in \cite{hitchin:duke}.

Let $G'$ be a reductive subgroup of $G$. A maximal compact subgroup of $G'$ 
is given by $H'=H\cap G'$ and we can take a compatible Cartan decomposition, 
in the sense that $\lieh'\subset \lieh$ and $\liem'\subset \liem$. Moreover, the 
isotropy representation of $H'$ is the restriction of the isotropy
representation $\Ad$ of $H$. We say that the structure group of a $G$-Higgs bundle 
$(E,\varphi)$ reduces to $G'$ when there is a reduction of the structure group of 
the underlying $H^\C$-bundle to $H'^\C$, given by a subbundle $E_\sigma$, and 
the Higgs field $\varphi\in H^0(X,E(\liem^\C)\ot K)$ belongs to $H^0(X,E_\sigma(\liem'^\C)\ot K)$.

\subsection{Stability of $G$-Higgs bundles}
\label{subsec:stability-G-Higgs-bundles}

There is a  notion of stability for $G$-Higgs
bundles (see \cite{GGM09}). To explain 
this  we  consider the   parabolic subgroups of $H^\C$ 
defined for $s\in i\lieh$ as
$$
P_s =\{g\in H^\C\ :\ e^{ts}ge^{-ts}\text{ is bounded as }t\to\infty\}.
$$
When $H^\C$ is connected every parabolic subgroup is conjugate to one of the 
form $P_s$ for some $s\in i\lieh$, but this is not the case necessarily when 
$H^\C$ is non-connected.
A Levi subgroup of $P_s$ is given by $L_s =\{g\in H^\C\ :\ \Ad(g)(s)=s\},$ 
and any Levi subgroup is given by $pL_sp^{-1}$ for $p\in P_s$. Their Lie algebras are given by
\begin{align*}
  \mathfrak{p}_s &=\{Y\in\lieh^\C\ :\ \Ad(e^{ts})Y\text{ is
    bounded as }t\to\infty\},\\
  \mathfrak{l}_s &=\{Y\in\lieh^\C\ :\ \ad(Y)(s)=[Y,s]=0\}.
\end{align*}

We  consider the subspaces
\begin{align*}
  &\liem_s=\{Y\in \liem^\C\ :\ \Ad(e^{ts})Y
  \text{ is bounded as}\;\; t\to\infty\}\\
  &\liem^0_{s}=\{Y\in \liem^\C\ :\ \Ad(e^{ts})Y=Y\;\;
  \mbox{for every} \;\; t\}.
\end{align*}

One has that $\liem_s$ is invariant under the action of $P_{s}$ and $\liem^0_{s}$
is invariant under the action of $L_{s}$. 
They are described in terms of root vectors by the following lemma.

\begin{lemma}\label{lemma:parabolic-eigenvectors} 
  Given $s \in i\lieh$, we have that
  \begin{align*}
    \liep_s &=\la Y\in \lieh^\C\ : \ \ad(s)Y=\lam_Y Y\;\;\mbox{for}\;\; \lam_Y\leq 0\ra, & \liel_s &=\la Y\in \lieh^\C\ : \ \ad(s)Y=0 \ra,\\
    \liem_s & =\la Y\in \liem^\C\ : \ \ad(s)Y=\lam_Y Y\;\;\mbox{for}\;\; \lam_Y\leq 0\ra, & \liem_s^0 & =\la Y\in \liem^\C\ : \ \ad(s)Y=0 \ra.
  \end{align*}
\end{lemma}

\begin{proof} We consider the endomorphism $\ad(s)$ and take $\{Y_\de\}_{\delta\in D \subset \C} \subset\liem^\C$, a basis of eigenvectors such that $\ad(s)Y_\de=\de Y_\de$. We have that
  $$\Ad(e^{ts})Y_\de=e^{\ad(ts_\chi)}Y_\de=\sum_{j=0}^{\infty} \frac{(\ad(t s))^j(Y_\de)}{j!}= \left(\sum_{j=0}^{\infty} \frac{(t\lam)^j}{j!}\right) Y_\de=e^{t\lam} Y_\de.$$
  Therefore, $Y_\de$ belongs to $\liem_s$ (resp. $\liem_s^0$) if and only if $\lam\leq 0$ (resp. $\lam=0$).
  By linearity, we obtain the result. \end{proof}

\begin{remark}
  The subalgebra $\liem_s$ is the non-compact part of the parabolic subalgebra 
  of $\lieg^\C$ defined by $s\in i\lieh$. Define 
  $\widetilde{\liep}_s=\{ Y\in \lieg^\C \st \Ad(e^{ts})Y \textrm{ is bounded as }
  t\to \infty \}.$  We have that $\liep_s = \widetilde{\liep}_s \cap \lieh^\C$ and 
  $\liem_s = \widetilde{\liep}_s \cap \liem^\C$. Analogously, define 
  $$\widetilde{\liel}_s=\{ Y\in \lieg^\C \st \Ad(e^{ts})Y = Y\;\; \mbox{for every} \;\; t\ \}.$$
  Then $\liel_s = \widetilde{\liel}_s \cap \lieh^\C$ and $\liem^0_s = \widetilde{\liel}_s \cap \liem^\C$.
\end{remark}

An  element $s\in i\lieh$ defines a character $\chi_s$ of $\liep_s$ since
$\langle s,[\liep_s,\liep_s]\rangle=0$.  Conversely, by the isomorphism 
$\left( \liep_s/ [\liep_s,\liep_s]\right)^* \cong \liez_{L_s}^*$, 
where $\liez_{L_s}$ is the centre of the Levi subalgebra $\liel_s$, a 
character $\chi$ of $\liep_s$ is given by an element in $\liez_{L_s}^*$, 
which gives, via the invariant metric, an element of 
$s_\chi\in \liez_{L_s}\subset i\lieh$. When $\liep_s\subset \liep_{s_\chi}$, 
we say that $\chi$ is an antidominant character of $\liep$. 
When  $\liep_s=\liep_{s_\chi}$ we say that 
$\chi$ is a strictly antidominant character. Note that for $s\in i\lieh$, 
$\chi_s$ is a strictly antidominant character of $\liep_s$.

\begin{remark}
  An approach based on root theory can be found in 
  \cite{GGM09}.  There, the antidominant characters are 
  also described in terms of fundamental weights.
\end{remark}

Let now $(E,\varphi)$ be a $G$-Higgs bundle over $X$, 
and let $s\in i\lieh$. Let $P_s$ be defined as above.
For $\sigma\in \Gamma(E(H^\C/P_s))$ a reduction of the structure group 
of $E$ from $H^\C$ to $P_s$, we define the degree relative to $\sigma$ and
$s$, or equivalently to $\sigma$ and $\chi_s$, as follows. When  
a real multiple $\mu \chi_s$ of the character exponentiates to a 
character $\tilde{\chi}_s$ of $P_s$, 
we compute the degree as 
$$\deg(E)(\sigma,s)= \frac{1}{\mu} \deg(E_\sigma(\tilde{\chi}_s)).$$
This condition is not always satisfied, but one shows 
(\cite[Sec. 4.6]{GGM09}) 
that the antidominant character can be expressed as a linear combination of 
characters of the centre and fundamental weights, $\chi_s=\sum_j z_j\mu_j +
\sum_k n_k \lambda_k$.  \cite[Lemma 2.4]{GGM09} states that 
there exists an integer multiple $m$ of the characters of the centre and the 
fundamental weights exponentiating to the group, so we can define the degree~as 
$$\deg(E)(\sigma,s)= 
\frac{1}{m}\left( \sum_j z_j \deg(E_\sigma(\widetilde{m \mu_j})) + 
  \sum_k n_k \deg(E_\sigma(\widetilde{m \lambda_k})) \right).$$
This value is independent of the expression of $\chi_s$ as sum of 
characters and the integer~$m$.

There is also a definition of the degree in terms of the 
curvature of connections using Chern--Weil theory.
This definition is  more natural when considering gauge-theoretic 
equations as we do below. For this, define $H_s=H\cap L_s$ and
$\lieh_s=\lieh\cap\liel_s$. 
Then $H_s$ is a maximal compact subgroup of $L_s$, so the inclusions
$H_s\subset L_s$ is a homotopy equivalence. Since the inclusion
$L_s\subset P_s$ is also a homotopy equivalence, given a reduction
$\sigma$ of the structure group of $E$ to $P_s$ one can
further restrict the structure group of $E$ to $H_s$ in a unique
way up to homotopy. Denote by $E'_{\sigma}$ the resulting $H_s$
principal bundle.
Consider now a connection $A$
on $E'_{\sigma}$ and let
$F_A\in\Omega^2(X,E'_{\sigma}(\lieh_s)$ be  its
curvature. Then $\chi_s(F_A)$ is a $2$-form on $X$ with
values in $i\R$, and 
\begin{equation}\label{degree-chern-weil}
  \deg(E)(\sigma,s):=\frac{i}{2\pi}\int_X \chi_s(F_A).
\end{equation}

We define the subalgebra $\lieh_{\ad}$ as follows. 
Consider the decomposition $ \lieh = \liez + [\lieh, \lieh] $, where $\liez$ 
is the centre of $\lieh$, and the isotropy representation 
$\ad= \ad:\lieh\to \End(\liem)$. Let $\liez'=\ker(\ad_{|\liez})$  and 
take $\liez''$ such that $\liez=\liez'+\liez''$. Define the subalgebra 
$\lieh_{\ad} := \liez'' + [\lieh, \lieh]$. The subindex $\ad$ denotes that 
we have taken away the part of the centre $\liez$ acting trivially via 
the isotropy representation $\ad$.

\begin{remark}\label{remark:fz'-Hermitian-type}
  For groups of Hermitian type, $\liez'=0$ since an element both in 
  $\liez$ and $\ker(\ad)$ belongs to the centre of $\lieg$, which is 
  zero, as $\lieg$ is semisimple. Hence $\lieh_{\ad} = \lieh$.
\end{remark}

With $L_s$ $\liem_s$ and $\liem^0_s$ defined as above. We have the following.

\begin{definition}
  \label{def:L-twisted-pairs-stability} 
  Let  $\alpha\in i\liez\subset\liez^\C$.
  We say that a $G$-Higgs bundle $(E,\varphi)$ is:

  $\alpha$-{\bf semistable} if for any $s\in i\lieh$ and any holomorphic reduction 
  $\sigma\in\Gamma(E(H^\C/P_s))$ such that $\varphi\in
  H^0(X,E_{\sigma}(\liem_s)\otimes K)$, 
  we have that $\deg(E)(\sigma,s)-\chi_s(\alpha)\geq 0.$

  $\alpha$-{\bf stable} if for any $s\in i\lieh_{\ad}$ and any holomorphic
  reduction $\sigma\in\Gamma(E(H^\C/P_s))$ such that $\varphi\in
  H^0(X,E_{\sigma}(\liem_s)\otimes K)$, we have that
  $\deg(E)(\sigma,s)-\chi_s(\alpha) > 0.$

  $\alpha$-{\bf polystable} if it is $\alpha$-semistable and for
  any $s\in i\lieh_{\ad}$ and any holomorphic reduction 
  $\sigma\in\Gamma(E(H^\C/P_s))$ 
  such that $\varphi\in  H^0(X,E_{\sigma}(\liem_s)\otimes K)$ and  
  $\deg(E)(\sigma,s)-\chi_s(\alpha)=0$, there is a holomorphic reduction 
  of the structure group $\sigma_L\in\Gamma(E_{\sigma}(P_s/L_s))$ to a Levi
  subgroup $L_s$ such that  $\varphi\in H^0(X,E_{\sigma_L}(\liem_s^0)\otimes K)
  \subset H^0(X,E_{\sigma}(\liem_s)\otimes K)$.
\end{definition}

\begin{remark}\label{remark:grupo-GL-Levi}
  We may define a real group $G_{L_s} = (L_s\cap H) \exp(\liem^0_s\cap
  \liem)$  with maximal compact subgroup the compact real 
  form $L_s\cap H$ of the complex group $L_s$ and $\liem^0_s\cap \liem$ 
  as isotropy representation. Thus, an $\alpha$- polystable $G$-Higgs 
  bundle reduces to an $\alpha$-polystable $G_{L_s}$-Higgs bundle 
  since $\varphi$ belongs 
  $H^0(X,E_{\sigma_L}(\liem_s^0)\otimes K)$. 
\end{remark}

\begin{remark}\label{twisted}
  We can  replace $K$ in the definition of $G$-Higgs bundle by any holomorphic
  line bundle $L$ on $X$. More precisely, a
  \textbf{$L$-twisted $G$-Higgs bundle} $(E,\varphi)$ consists of a principal
  $H^\C$-bundle $E$, and a holomorphic section $\varphi\in H^0(X,E(\liem^\C)\ot
  L)$. We reserve the name $G$-Higgs bundle for the $K$-twisted case. 
  The stability criteria  are as in Definition 3.5, replacing $K$ by $L$.
\end{remark}

\begin{remark}
  For $G$ semisimple, the notion of $\al$-stability with $\al\neq 0$ only makes
  sense for groups of Hermitian type, since $\al$ belongs to the centre of
  $\lieh$, which is not zero if and only if the centre of a maximal compact
  subgroup $H$ is non-discrete, i.e., if $G$ is of Hermitian type. For reductive
  groups which are not of Hermitian type, $\al$-stability makes sense, but there
  is only one value of $\al$ for which the  condition is not void. This value is
  fixed by the topology of the principal bundle (see
  \cite{GO14} for details).
\end{remark}

\begin{remark}
  When $H^\C$ is connected, as mentioned above, every parabolic subgroup of
  $H^\C$ is conjugate to one of the form $P_s$ for $s\in i\lieh$. In this 
  situation, we can formulate the stability conditions in Definition 
  \ref{def:L-twisted-pairs-stability} in terms of any parabolic subgroup 
  $P\subset H^\C$, replacing $s$ by $s_\chi$, for any antidominant character 
  $\chi$ of $\lieh^\C$.
\end{remark}

Two $G$-Higgs bundles $(E,\varphi)$ 
and $(E',\varphi')$ are isomorphic if there is an isomorphism $f:E\to E'$ such 
that $\varphi'=f^*\varphi$, where $f^*$ is the map $E(\liem^\C)\ot K \to E'(\liem^\C)\ot
K$  induced by $f$.

The {\bf moduli space of $\alpha$-polystable $G$-Higgs bundles} 
$\cM^\alpha(G)$  is defined as the set of isomorphism classes of 
$\alpha$-polystable  $G$-Higgs bundles on $X$. When $\alpha=0$ 
we simply denote  $\cM(G):= \cM^0(G)$.

\begin{remark}
  Similarly, we can define the {\bf moduli space of $\alpha$-polystable 
    $L$-twisted $G$-Higgs bundles} which will be denoted  by $\cM^\alpha_L(G)$.
\end{remark}

These moduli spaces  have the structure of a complex analytic
varieties,  as one can see by the standard slice method, which gives local 
models via the so-called Kuranishi map (see, e.g., \cite{kobayashi}).  
When $G$ is algebraic and under fairly general conditions, 
the moduli spaces $\cM^\al(G)$ can be constructed by geometric invariant 
theory and hence are complex algebraic varieties. The  work of 
Schmitt \cite{schmitt05,schmitt08} deals with the construction 
of the moduli space of $L$-twisted $G$-Higgs bundles for $G$ a real reductive 
Lie group. This construction generalizes the constructions of the moduli space 
of $G$-Higgs bundles done by Ramanathan \cite{ramanathan96} when $G$ is
compact,  and by Simpson \cite{simpson94,simpson95} when $G$ 
is a complex reductive algebraic (see also \cite{nitsure} for $G=\GL(n,\C)$).

The notion of stability emerges from the study of the
Hitchin equations.  The equivalence between the existence of solutions to 
these equations and the $\al$-polystability of Higgs bundles is known as 
the {\bf Hitchin-Kobayashi correspondence}, which we state below.

\begin{theorem}\label{theo:hk-twisted-pairs}
  Let $(E,\varphi)$ be a $G$-Higgs bundle over a Riemann surface $X$ with volume
  form $\omega$. Then $(E,\varphi)$ is $\alpha$-polystable if and only if 
  there exists a reduction $h$ of the structure group of $E$ from $\H^\C$ to
  $H$, that is a smooth section of $E(H^\C/H)$,  such that
  \begin{equation}\label{eq:Hitchin-Kobayashi-h}
    F_h - [\varphi,\tau_h(\varphi)]=-i\al \omega\\
  \end{equation}
  where $\tau_h:\Omega^{1,0}(E(\liem^\C))\to \Omega^{0,1}(E(\liem^\C))$ is the
  combination of  the anti-holomorphic involution in $E(\liem^\C)$ defined 
  by the compact real form at each point determined by $h$  and the 
  conjugation of $1$-forms, and $F_h$ is the 
  curvature of the unique $H$-connection compatible with the holomorphic 
  structure of $E$ (the Chern connection). 
\end{theorem}

This theorem was proved by Hitchin in the case of $\SL(2,\C)$, by Simpson when
$G$ is complex, and in \cite{BGM03,GGM09} for a general 
reductive real Lie group $G$

\begin{remark}\label{theo:-twisted-hk-twisted-pairs}
  There is a theorem similar  to Theorem \ref{theo:hk-twisted-pairs} for
  $L$-twisted $G$-Higgs bundles (see Remark \ref{twisted} and
  \cite{GGM09}) 
  for an arbitrary line bundle $L$. If $(E,\varphi)$ is such a pair,
  one fixes a Hermitian metric $h_L$ on $L$, and looks for
  a reduction of structure group $h$  of $E$ from $H^\C$ to $H$ satisfying

  \begin{equation}\label{eq:twisted-Hitchin-Kobayashi-h}
    F_h - [\varphi,\tau_h(\varphi)]\omega=-i\al \omega,
  \end{equation} 
  where now 
  $\tau_h:\Omega^0(E(\liem^\C)\otimes L)\to \Omega^0(E(\liem^\C)\otimes L)$ is 
  the combination of the anti-holomorphic involution in $E(\liem^\C)$ defined
  by the compact real form at each point determined by $h$  and the 
  metric $h_L$.
\end{remark}

\subsection{Higgs bundles and representations}\label{higgs-reps}

Fix a base point $x\, \in\, X$.
A \textbf{representation} of $\pi_1(X,x)$ in
$G$ is  a homomorphism $\pi_1(X,x) \,\longrightarrow\, G$.
After fixing a presentation of $\pi_1(X,x)$, the set of all such homomorphisms,
$\Hom(\pi_1(X,x),\, G)$, can be identified with the subset
of $G^{2g}$ consisting of $2g$-tuples
$(A_{1},B_{1}, \cdots, A_{g},B_{g})$ satisfying the algebraic equation
$\prod_{i=1}^{g}[A_{i},B_{i}] \,=\, 1$. This shows that
$\Hom(\pi_1(X,x),\, G)$ is an algebraic variety.

The group $G$ acts on $\Hom(\pi_1(X,x),G)$ by conjugation:
$$
(g \cdot \rho)(\gamma) \,=\, g \rho(\gamma) g^{-1}\, ,
$$
where $g \,\in\, G$, $\rho \,\in\, \Hom(\pi_1(X,x),G)$ and
$\gamma\,\in \,\pi_1(X,x)$. If we restrict the action to the subspace
$\Hom^+(\pi_1(X, x),\,G)$ consisting of reductive representations,
the orbit space is Hausdorff.  We recall that a \textbf{reductive representation}
is one whose composition with the adjoint representation in $\mathfrak g$
decomposes as a direct sum of irreducible representations.
This is equivalent to the condition that the Zariski closure of the
image of $\pi_1(X,x)$ in $G$ is a reductive group. Define the
{\bf moduli space of representations} of $\pi_1(X,x)$ in $G$ to be the orbit space
$$
\mathcal{R}(G) = \Hom^{+}(\pi_1(X,x),G) /G.
$$
This is a real algebraic variety.
For another point $x'\, \in\, X$, the fundamental groups
$\pi_1(X,x)$ and $\pi_1(X,x')$ are identified by an isomorphism unique up to
an inner automorphism. Consequently, $\mathcal{R}(G)$ is independent of the choice of
the base point $x$.

Given a representation $\rho\colon\pi_{1}(X,x) \,\longrightarrow\,
G$, there is an associated flat principal $G$-bundle on
$X$, defined as
$$
E_{\rho} \,=\, \widetilde{X}\times_{\rho}G\, ,
$$
where $\widetilde{X} \,\longrightarrow\, X$ is the universal cover
and $\pi_{1}(X, x)$ acts
on $G$ via $\rho$.
This gives in fact an identification between the set of equivalence classes
of representations $\Hom(\pi_1(X),G) / G$ and the set of equivalence classes
of flat principal $G$-bundles, which in turn is parametrized by
the (non-abelian) cohomology set $H^1(X,\, G)$. We have the following.

\begin{theorem}\label{na-Hodge}
  Let $G$ be a semisimple real Lie group. Then there is a homeomorphism
  $\mathcal{R}(G) \,\cong\, \cM(G)$. 
\end{theorem}

The proof of Theorem \ref{na-Hodge} is the combination of Theorem
\ref{theo:hk-twisted-pairs} and the following  theorem of Corlette \cite{corlette},
also proved by Donaldson \cite{donaldson} when $G=\SL(2,\C)$.

\begin{theorem}\label{corlette}
  Let $\rho$ be a representation of $\pi_1(X)$ in $G$ with corresponding
  flat $G$-bundle $E_\rho$. Let $E_\rho(G/H)$ be the associated $G/H$-bundle.
  Then the existence of a harmonic section of $E_\rho(G/H)$ is equivalent to
  the reductiveness of $\rho$. 
\end{theorem}

\begin{remark}
  Theorem \ref{na-Hodge} can be extended to reductive groups replacing 
  $\pi_1(X)$ by its universal central extension.
\end{remark}

\section{Hermitian groups, Toledo invariant and  Milnor-Wood inequality}
\label{sec:al-milnor-wood-inequality}

In this section we will assume that $G$ is a connected, non-compact
real simple Lie group of Hermitian type with finite centre (see Section
\ref{cayley-transform} for definition), and $X$ is
a compact Riemann surface of genus $g$. We fix a maximal compact
subgroup $H\subset G$, with Cartan decomposition $\lieg=\lieh +\liem$.

Let $(E,\varphi)$ be a $G$-Higgs bundle over $X$.  The decomposition
$\liem^\C=\liem^++\liem^-$ gives a vector bundle decomposition $E(\liem^\C)=
E(\liem^+) \oplus E(\liem^-)$ and hence the Higgs field has two components:
$$
\varphi=(\varphi^+, \varphi^-)\in H^0(X,E(\liem^+)\otimes K)\oplus
H^0(X,E(\liem^-)\otimes K)= H^0(X,E(\liem^\C)\otimes K).
$$

When the group $G$ is a classical group, or more generally when
$H$ is a classical group, 
it is  useful to take the standard representation of $H^\C$ to describe a 
$G$-Higgs bundle in terms of associated vector bundles. 
This is the approach taken in \cite{BGG03,BGG06,BGG15,GGM13}. 

\subsection{Toledo invariant}
\label{sec:Toledo-and-topological-class}

Let $(E,\varphi)$ be a $G$-Higgs bundle. Consider the Toledo character $\chi_T$
defined in Section \ref{sec:Toledo-character}. Maybe up to an integer
multiple, $\chi_T$ lifts to a character $\tilde{\chi}_T$ of $H^\C$. Let
$E(\tilde{\chi}_T)$ be the line bundle associated to $E$ via the
character $\tilde \chi_T$.

\begin{definition}\label{def:Toledo-invariant}
  We define the \textbf{Toledo invariant} $\tau$ of $(E,\varphi)$ 
  by
  $$
  \tau=\tau(E):=\deg(E(\tilde{\chi}_T)).
  $$
\end{definition}

If $\tilde \chi_T$ is not defined, but only $\tilde \chi_T^q$, one must replace the definition by $\frac1q\deg E(\tilde{\chi}^q_T).$

We denote by $\cM^\alpha_\tau(G)$ the subspace of $\cM^\alpha(G)$ 
corresponding to  $G$-Higgs bundles whose Toledo invariant equals $\tau$.
For $\alpha=0$ we simplify our notation setting  $\cM_\tau(G):= \cM^0_\tau(G)$

The following proposition relates our Toledo invariant to the usual Toledo invariant of a representation, first defined in \cite{toledo}.
\begin{proposition}\label{prop:toledo-rep}
  Let $\rho:\pi_1(X) \to G$ be reductive and let $(E,\varphi)$ be the corresponding
  polystable $G$-Higgs bundle given by Theorem \ref{na-Hodge}. Let $f:\tilde X\to G/H$ be the corresponding harmonic metric. Then
  $$ \tau(E) = \frac1{2\pi} \int_X f^*\omega , $$
  where $\omega$ is the K\"ahler form of the symmetric metric on $G/H$ with
  minimal holomorphic sectional curvature $-1$, computed
  in Lemma \ref{lemma:Toledo-Kahler-form}.

  In particular, $\tau(E)$ is the Toledo invariant of $\rho$.
\end{proposition}
\begin{proof}
  
  The harmonic metric $f$ defines a solution $h$ to the Hitchin equations
  (\ref{eq:Hitchin-Kobayashi-h}), that is a reduction of structure group of $E$ to $H$. 
  Write $\Phi=\varphi-\tau_h(\varphi)=\Phi_xdx+\Phi_ydy$ in local coordinates,
  where $\tau_h$ is defined in Theorem \ref{theo:hk-twisted-pairs}. 
  Then $[\varphi,\tau_h(\varphi)]=[\Phi_x,\Phi_y]dx\land dy$. From the Hitchin equation
  $F_h=[\varphi,\tau_h(\varphi)]$ and 
  Lemma \ref{lemma:Toledo-Kahler-form},
  $$ \frac i{2\pi} \chi_T(F_h) = \frac i{2\pi} \chi_T([\Phi_x,\Phi_y])dx\land dy
  =  \frac1{2\pi} \omega(\Phi_x,\Phi_y)dx\land dy.$$
  Since $\Phi=df$, the expected formula follows by integration.
\end{proof}

The Toledo invariant  is related to the topological class of the
bundle $E$ defined as an element of $\pi_1(H)$. To explain this, assume that
$H^\C$ is connected. The topological classification of $H^\C$-bundles $E$ on
$X$  is given by a characteristic class $c(E)\in \pi_1(H^\C)$ as follows. 
From the exact sequence 
$$
1\to \pi_1(H^\C) \to \widetilde{H^\C} \to H^\C \to 1
$$  
we obtain a long
exact  sequence in cohomology and, in particular, the connecting homomorphism 
\begin{equation}
  H^1(X,\underline{H}^\C)\xra{c} H^2(X,\pi_1(H^\C)),\label{eq:1}
\end{equation}
where $\underline{H}^\C$ is the sheaf of local holomorphic functions in $X$  
with 
values in $H^\C$. The cohomology set $H^1(X,\underline{H}^\C)$ 
(not necessarily  a group since $H^\C$ is  in general
not abelian) parametrizes isomorphism classes of principal $H^\C$-bundles 
over $X$. On the other hand, since $\dim_\R X=2$, by the universal 
coefficient theorem  and the fact that the fundamental group of a 
Lie group is abelian,    $H^2(X,\pi_1(H^\C))$ is isomorphic to
$\pi_1(H^\C)$. Moreover, $\pi_1(H^\C)\cong \pi_1(H)\cong\pi_1(G)$ 
since $H$ is a deformation retract for
both $H^\C$ and $G$. This map thus associates a topological invariant 
in $\pi_1(H)$  to any $G$-Higgs bundle on $X$.

By the relation between the fundamental group and the centre of a Lie
group, the topological class in $\pi_1(H)$ is of special interest when
$H$ has a non-discrete centre, i.e., when $G$ is of Hermitian type. In
this case, $\pi_1(H)$ is isomorphic to $\Z$ plus possibly a torsion
group (among the classical groups, $\SO_0(2,n)$ is the only one with
torsion). Very often (see for example \cite{BGG06}), the Toledo
invariant of a $G$-Higgs bundle $(E,\varphi)$  is defined as the projection
of $c(E)$ defined by (\ref{eq:1}) on the torsion-free part, $\Z$. The
general relation is the following.

\begin{proposition}\label{prop:Milnor-Wood-mod-invariant}
  Let $(E,\varphi)$ be a $G$-Higgs bundle, and $d\in \Z$ the projection on the
  torsion-free part of the class $c(E)$ defined by (\ref{eq:1}). Then
  $d$ is related to the Toledo invariant by
  $$ 
  \tau = \frac{d}{q_T}.$$
\end{proposition}

\begin{proof}
  The character $\tilde \chi_T:H^\C\to\C^*$ induces a morphism
  $$ 
  (\tilde \chi_T)_*: \pi_1(H^\C)\longrightarrow\pi_1(\C^*)=\Z,
  $$
  defined by $(\tilde \chi_T)_*(\gamma)=\tau(E)$, 
  where $E$ is the $H^\C$-bundle corresponding to  $\gamma\in\pi_1(H^\C)$ under the  
  bijection of $H^1(X,H^\C)$ with $\pi_1(H^\C)$. This map  must be a multiple 
  of the projection $\pi_1(H^\C)\to\Z$ on the torsion-free part.

  We compute this multiple by considering the image by $\tilde \chi_T$ of
  a loop generating $\pi_1(H^\C)$. Recall that $H^\C=[H^\C,H^\C]\ts_D
  Z_0^\C$, where $Z_0^\C$ is the identity component of the centre of
  $H^\C$ and $D=[H^\C,H^\C]\cap Z_0^\C$. We have seen in the proof of
  Proposition \ref{prop:exponentiation-of-Toledo} that the torsion-free part of $\pi_1(H^\C)$ is
  generated by the loop $\exp(2\pi\frac{o_J}{\ell} J\theta)$,
  and hence, since
  $ (\chi_T)_* (\frac{o_J}{\ell} J) = \frac{o_J}{\ell} \frac{\dim\liem}{N} i$,
  the result follows.
\end{proof}

\subsection{Milnor-Wood inequality}
\label{sec:milnor-wood-inequality}

In this section we give a bound of the Toledo invariant for an
$\alpha$-semistable $G$-Higgs bundle $(E,\varphi^+,\varphi^-)$ involving the ranks of $\varphi^+$
and $\varphi^-$, leading to the familiar Milnor-Wood inequality.

Definition \ref{definition-rank} gives the ranks of $\varphi^+$  and $\varphi^-$ 
at a point $x\in X$. The space of elements of $\liem^\pm$ with rank at most $\rho$ is an algebraic subvariety of $\liem^\pm$, so the ranks of $\varphi^+$ and $\varphi^-$ are the same at all points of $X$ except a finite number of points where it it smaller. We therefore have a well defined notion of rank of $\varphi^+$ and $\varphi^-$:
\begin{definition}\label{def:rank-Higgs-field}
  The generic value on $X$ of the rank of $\varphi^+$ is called the rank of
  $\varphi^+$ and denoted $\rk \varphi^+$. Analogously we define the rank $\rk \varphi^-$ of
  $\varphi^-$.
\end{definition}

The main result of this section is the following.
\begin{theorem}\label{theo:ineq-rk} Let $\al\in i\liez$ such that $\al=i\lambda
  J$ for $\lambda\in\R$. Let $(E,\varphi^+,\varphi^-)$ be an $\al$-semistable $G$-Higgs
  bundle. Then, the Toledo invariant of $E$ satisfies:
  $$-\rk(\varphi^+)(2g-2)-\left(\frac{\dim \liem}{N}-\rk(\varphi^+)\right)\lambda \leq
  \tau \leq \rk(\varphi^-)(2g-2)-\left(\frac{\dim \liem}{N}-\rk(\varphi^-)\right)\lambda,$$ where $N$
  is the dual Coxeter number. In the tube case, this simplifies to:
  $$-\rk(\varphi^+)(2g-2)-(r-\rk(\varphi^+))\lambda \leq \tau \leq \rk(\varphi^-)(2g-2)-(r-\rk(\varphi^-))\lambda.$$
\end{theorem}

The rest of this subsection is devoted to the proof of the theorem. We will focus on $\varphi^+$ to prove the left inequality, the right one is a consequence of the same proof for $\varphi^-$.

We will apply the semistability hypothesis to a carefully defined parabolic reduction of $E$. Recall that a nilpotent endomorphism $U$ in a vector space has a unique decreasing filtration $W=(W_k)$, such that $U(W_k)\subset W_{k+2}$ and $U^k$ induces an isomorphism on the graded spaces: $W_{-k}/W_{-k+1}\to W_k/W_{k+1}$. If one completes $U$ into a representation $(S,U,V)$ of $\fsl(2,\R)$, then $W_k$ is the sum of the eigenspaces of $S$ for the eigenvalues $\lambda\geq k$.

Now consider a point $x\in X$. To begin with, fix a trivialization of $E$, so that we can consider $\varphi^+_x$ as an element of $\liem^+$. Then $U=\ad \varphi^+_x$ acts on $\lieg^\C=\lieh^\C\oplus\liem^\C$, so from the associated $W$ filtration, we can define
$$ \liep_x = W_0 \cap \lieh^\C. $$
This $W$ filtration is very simple, because $U(\liem^+)\subset \lieh^\C$, $U(\lieh^\C)\subset \liem^- \subset \ker U$, so actually
\begin{equation}
  \liep_x = \ker U|_{\lieh^\C} + \im U|_{\liem^-}.\label{eq:2}
\end{equation}
The element $\varphi^+_x$ can be completed in a $\fsl_2$-triple $(h,\varphi^+_x,v)$ such that $h\in \lieh^\C$ (and $v\in \liem^-$). Then a $\fsl_2$-triple associated to $U$ is $(\ad h,U,\ad v)$. From the definition of $W_0$ using the eigenspaces of $\ad h$, it follows now that $\liep_x$ is a parabolic subalgebra of $\lieh^\C$.

One can make things more explicit: choose a system of strongly
orthogonal roots $\{\gamma_1,\ldots,\gamma_r\}$ as in Section 
\ref{sec:GHTandRT},
suppose $\rk \varphi^+_x=r'$, then $\varphi^+_x$ is conjugate to the element
\begin{equation}
  u = e_{\gamma_1} + \cdots + e_{\gamma_{r'}} \in \liem^+.\label{eq:4}
\end{equation}
Then we choose
\begin{align}
  v&=e_{-\gamma_1}+\cdots +e_{-\gamma_{r'}}\label{eq:9}\\
  h&=h_{\gamma_1}+\cdots +h_{\gamma_{r'}}\label{eq:10}
\end{align}
so that we obtain the description
$$ \liep = \liet^\C \oplus \underset{\alpha\in \Delta_C, \alpha(h)\geq 0}\oplus \lieg_\alpha. $$
Denote by $P$ the corresponding parabolic subgroup of $H^\C$.

All the previous description was just at the point $x$. But since the
construction 
is canonical (indeed, the $W$ filtration depends only on $\varphi^+_x$), we obtain
a  reduction of the structure group of $E$ to $P$ on the open set where the
rank  of $\varphi^+_x$ is $r'=\rk \varphi^+$. The description (\ref{eq:2}) with 
$U=\ad \varphi^+$ shows that the reduction extends over the singular points 
(since the kernel or the image of a morphism between holomorphic bundles do). 
Therefore $\varphi^+$ defines a global reduction $\sigma$ to $P$ of the $H^\C$-bundle $E$.

To apply the semistability criterion, the second ingredient is an antidominant 
character of $\liep$. We consider
\begin{equation}
  \chi_{r'} = \gamma_1 + \cdots + \gamma_{r'}\label{eq:5}
\end{equation}
to define a character of $\liep$ by
\begin{equation}
  \label{eq:3}
  \chi = \chi_T - \chi_{r'}.
\end{equation}
Let $s_\chi\in \liet^\C$ be dual to $\chi$ via the invariant product. Then:
\begin{lemma}\label{lem:antidom-char}
  For $\alpha\in \Delta$ one has
  $$ \frac1{\langle\gamma,\gamma\rangle}\alpha(s_\chi) =
  \begin{cases}
    1-\frac12 \alpha(h) & \text{ if }\alpha\in \Delta_Q^+, \\
    -\frac12 \alpha(h)  & \text{ if }\alpha\in \Delta_C, \\
    -1-\frac12 \alpha(h)& \text{ if }\alpha\in \Delta_Q^-.
  \end{cases}$$
  In particular, $s_\chi$ defines a strictly  antidominant character of $\liep$, 
  such that $\liem^- \subset \liem_{s_\chi}$ and $u\in \liem^0_{s_\chi}$.
\end{lemma}
\begin{proof}
  Recall that the dual of $\gamma_i$ is $s_{\gamma_i}=\frac{\langle\gamma_i,\gamma_i\rangle}2 h_{\gamma_i}$, so
  the dual of $\chi$ is
  \begin{equation*}
    s_\chi = s_{\chi_T} - \frac{\langle\gamma_i,\gamma_i\rangle}2 h = \langle\gamma_i,\gamma_i\rangle \big( -iJ - \frac h2 \big).
  \end{equation*}
  The formula in the lemma follows immediately. In particular, $s_\chi$
  defines a strictly antidominant character for $\liep$, since on $\lieh^\C$ it
  coincides (up to a positive constant) with $-h$. Using $[h,u]=2u$ and $u\in
  \liem^+$, we deduce that $[s_\chi,u]=0$ so $u$ is actually in the Levi part
  $\liem^0_s\subset \liep$. Moreover, since all the eigenvalues of $\ad h$ are in
  $\{-2,\ldots,2\}$, we have $-1-\frac12 \alpha(h)\leq 0$ for all $\alpha$,
  and hence  $\liem^-\subset \liem_{s_\chi}$.
\end{proof}

The last algebraic ingredient needed for the proof of the theorem is the remark that the choice of $r'$ strongly orthogonal roots $\{\gamma_1,\ldots,\gamma_{r'}\}$ defines a `subtube'. This will also give a useful interpretation of $\chi$ in terms of the subtube. Indeed, we can define, following the notations of Section \ref{appendix:restricted-root-theory}, (we will give a more intrinsic construction later)
$$ C_{r'} = \underset{1\leq i,j\leq r'}\cup C_{ij}, \quad
Q_{r'} = \underset{1\leq i,j\leq r'}\cup Q_{ij}, $$
and
$$ \lieh_{r'}^\C = \langle h_\gamma, \gamma\in Q_{r'}\rangle \oplus \underset{\alpha\in C_{r'}}\oplus \lieg_\alpha, \quad
\liem_{r'}^\pm = \underset{\alpha\in Q_{r'}}\oplus \lieg_\alpha. $$
Intersecting with $\lieg$ we obtain well-defined real forms $\lieh_{r'}$ and $\liem_{r'}$.

This construction is illustrated in the case of $\SU(p,q)$ by the matrix
$$ \left(\begin{array}{cc|ccc}
    0 &  &  &  &  \\
    & \lieh^\C_{r'} & \liem^+_{r'} &  &  \\  \hline
    & \liem^-_{r'} & \lieh^\C_{r'} &  & \\
    &  & & 0 & \\
    &  &  &  & 0 \\
  \end{array}\right).$$
In general, $\lieg_{r'}=\lieh_{r'}\oplus\liem_{r'}$ has all the properties of a symmetric pair of tube type of rank $r'$, having $\{\gamma_1,\ldots,\gamma_{r'}\}$ as a set of strongly orthogonal roots. The character $\chi_{r'}$ defined in (\ref{eq:5}) clearly identifies to the Toledo character of this subtube. Let $H_{r'}\subset H$ and $G_{r'}\subset G$ the corresponding subgroups. All the algebraic study of Section \ref{sec:GHTandRT} applies: in particular we have a determinant $\det_{r'}$ on $\liem_{r'}^+$, a Toledo character $\chi_{r'}$ lifting to a Toledo character $\tilde \chi_{r'}$ of $H_{r'}^\C$, so that for $h\in H_{r'}^\C$ and $m\in \liem_{r'}^+$, we have (maybe taking some power if the lifting does not exist)
\begin{equation}
  \label{eq:6}
  \det\nolimits_{r'}(\Ad(h)m) = \tilde \chi_{r'}(h) \det\nolimits_{r'}(m).
\end{equation}

Of course, coming back to our bundle $E$, we must see that this construction
is intrinsic so that it defines bundles over the surface $X$. It is easily
checked that $\liem_{r'}^+$ coincides with the step $W_2$ of the $W$ filtration
of $\ad \varphi^+$ (which equals the eigenspace of $\ad h$ for the eigenvalue $2$
because this is the highest eigenvalue). Moreover,
the Levi factor $L$ of $P$ acts on $\liem_{r'}^+$ by $\Ad$, so we can can
consider 
$$
L'=\{ g\in L\;:\; \Ad(g)|_{\liem_{r'}^+}=1 \},
$$ 
and we obtain a faithful action of
$$ \tilde L = L/L' $$
on $\liem_{r'}^+$; the group $\tilde L$ identifies to $H_{r'}^\C$ (so it is more
intrinsically defined as a quotient 
rather than a subgroup).

All this depends canonically on $\varphi^+$, so produces a
pair $$(E_{H_{r'}^\C},\varphi^+)$$ of a holomorphic $H_{r'}^\C$ principal bundle
$E_{H_{r'}^\C}$ and a section 
$\varphi^+\in H^0(X,E_{H_{r'}^\C}(\liem_{r'}^+)\otimes K)$. 
Because $\rk \varphi^+=r'=\rk G_{r'}/H_{r'}$, the determinant $\det_{r'}$ produces a non-zero section
$$ \det\nolimits_{r'} \varphi^+ \in H^0(X,E_{H_{r'}^\C}(\tilde \chi_{r'})\otimes K^{r'}), $$
which implies
\begin{equation}
  \label{eq:7}
  \deg E_{H_{r'}^\C}(\tilde \chi_{r'}) + r' (2g-2) \geq 0.
\end{equation}

We have now all the ingredients needed for the proof of Theorem
\ref{theo:ineq-rk}. 

\begin{proof}[of Theorem \ref{theo:ineq-rk}]
  We constructed from $\varphi^+$ a reduction $\sigma$ of $E$
  from $H^\C$ to $P=P_{s_\chi}$. Call $E_\sigma$ the reduced $P$-bundle. From Lemma
  \ref{lem:antidom-char}, we have $\varphi\in H^0(X,E_\sigma(\liem_{s_\chi})\otimes K)$.  The
  semistability condition gives
  \begin{equation}\label{eq:semistability-condition}
    \deg(E)(\sigma,s_\chi)-\langle\al,s_\chi\rangle \ \geq 0.
  \end{equation}

  The character $\chi$ lifts to a character of $P$ given by
  $$\tilde{\chi}=\tilde{\chi}_T \tilde{\chi}_{r'}^{-1}.$$
  (Again, it may be necessary to take some power). Therefore
  \begin{align*}
    \deg(E)(\sigma,s)
    &= \deg(E_P(\tilde{\chi}_T))-\deg(E_P(\tilde{\chi}_{r'})), \\
    &= \tau - \deg(E_P(\tilde{\chi}_{r'})) \\
    &\leq \tau + r' (2g-2)
  \end{align*}
  where at the last line, we have used (\ref{eq:7}). So
  (\ref{eq:semistability-condition}) becomes
  \begin{equation}\label{eq:MW-1}
    \tau \geq -r' (2g-2) + \langle\alpha,s_\chi\rangle .
  \end{equation}
  From (\ref{eq:15}) we obtain
  $$ \langle\alpha,s_\chi\rangle=(\chi_T-\chi_{r'})(\lambda iJ) = -\lambda\big(\frac{\dim\liem}N - r' \big). $$
\end{proof}

As an immediate corollary of Theorem \ref{theo:ineq-rk} we have the following.

\begin{proposition}\label{prop:inequality-tau-r-lambda}
  Let $\al=i\lambda J$. If $(E,\be,\ga)$ is an $\al$-semistable $G$-Higgs bundle
  with Toledo invariant $\tau$, we have that 
  $$
  |\tau|\leq \begin{cases}
    (2g-2 + \lambda)\rk(G/H) - \frac{\dim \liem}{N}\lambda & \text{ if } \lambda> -(2g-2),\\
    -\frac{\dim \liem}{N}\lambda & \text{ if } \lambda\leq -(2g-2).
  \end{cases}
  $$
\end{proposition}

The case $\al=0$ is of special interest because of the relation of the  moduli space of polystable 
$G$-Higgs bundles over a Riemann surface $X$  with  the moduli space of
representations  of $\pi_1(X)$ into $G$ as explained in Section \ref{higgs-reps}.  
In  this situation,  we will simply talk about 
\textbf{stability} of a $G$-Higgs bundle,  meaning $0$-stability, and 
analogously for \textbf{semistability}  and \textbf{polystability}. 
When $\alpha=0$, we have that $\lambda=0$ in Theorem \ref{theo:ineq-rk}, and
obtain the following.

\begin{theorem}\label{0-theo:ineq-rk} 
  Let $(E,\varphi^+,\varphi^-)$ be a semistable $G$-Higgs bundle. Then the
  Toledo invariant $\tau$ satisfies 
  $$
  -\rk(\varphi^+)(2g-2) \leq \tau \leq \rk(\varphi^-)(2g-2).
  $$ 

  In particular,  we obtain the familiar Milnor-Wood inequality:
  $$
  |\tau|\leq \rk(G/H)(2g-2),
  $$
  and the equality holds if and only if $\varphi^+$ (resp. $\varphi^-$) is regular at each point in the case $\tau<0$ (resp. $\tau>0$).
\end{theorem}
\begin{proof}
  There only remains to prove the last statement on the equality case, but this is an immediate consequence of the proof of Theorem \ref{theo:ineq-rk}: for simplicity, restrict to the case $\tau<0$, then equality in the Milnor--Wood inequality implies equality in (\ref{eq:7}), which implies that $\det_r \varphi^+$ does not vanish. The converse is immediate.
\end{proof}

Theorem \ref{0-theo:ineq-rk} 
was proved on a case by case 
basis for the classical groups \cite{hitchin87,gothen,BGG03,BGG06,BGG15,GGM13}. 
In these 
references, the bound  given is for the integer 
$d\in \pi_1(H)\cong\pi_1(H^\C)\cong\Z$ associated 
naturally to the $H^\C$-bundle $E$. This differs
from the Toledo invariant by a rational multiple.
From Table \ref{tab:exponent-tubetype} 
and Proposition \ref{prop:Milnor-Wood-mod-invariant} combined with 
Theorem \ref{0-theo:ineq-rk} we obtain  the Milnor--Wood  inequalities given
in \cite{BGG06} for the classical Hermitian groups.
Our intrinsic general approach covers of course the exceptional
groups and quotients and covers of classical groups that have not
been studied previously. 

A polystable $G$-Higgs bundle
$(E,\varphi)$ is, by  Theorem \ref{na-Hodge},  in correspondence with a reductive representation
$\rho:\pi_1(X)\to G$, and from Proposition  \ref{prop:toledo-rep}
the Toledo invariant of $(E,\varphi)$ coincides
with the Toledo invariant of a representation of the fundamental 
group in $G$. 
In the context of representations the inequality $|\tau|\leq \rk(G/H)(2g-2)$,
goes back to Milnor \cite{milnor}, who studies the case $G=\PSL(2,\R)$,  and was
proved  in various cases in  \cite{wood,dupont,DT87,CO03}, and in general 
in \cite{BIW10}.
We should point out that the Higgs bundle 

approach gives the Milnor--Wood inequality for an arbitrary representation, 
as the other approaches do,  since such a representation   
can always be deformed to a reductive one. 

\begin{remark}
  Lemma \ref{lemma:Toledo-Kahler-form} and Proposition \ref{prop:toledo-rep} 
  indeed provide the way to translate the Milnor-Wood inequality for Higgs bundles into the inequality for representations, for any group of Hermitian type with finite centre, in a classification-independent way. Translating the Milnor-Wood inequality for representations into that for Higgs bundles is the aim of \cite{HO11}, where it is done for the groups $\SU(p,q)$ and $\Sp(2n,\R)$, or a matrix Lie group admitting a so-called admissible representation. 
\end{remark}

\subsection{Involutions and the Toledo invariant}\label{anti-involutions}

Let $G^\C$ be a complex  Lie group and  let $\Aut(G^\C)$ be the group
of its holomorphic automorphisms. Let $\Int(G^\C)\subset \Aut(G^\C)$ be 
the subgroup of inner automorphisms. The group of outer automorphisms of $G^\C$ is defined as
$$
\Out(G^\C):=\Aut(G^\C)/\Int(G^\C).
$$

We thus have a sequence
\begin{equation}\label{outer-extension-group}
  1 \lra \Int(G^\C)\lra \Aut(G^\C) \lra \Out(G^\C)\lra 1.
\end{equation}

A real Lie subgroup $G$ of the underlying real Lie 
group to $G^\C$  is a {\bf real form} of $G^\C$ if it is the fixed point set of a 
conjugation (i.e. an anti-holomorphic involution)  $\sigma$ of $G^\C$. 

Assume now that  $G^\C$ is semisimple.  Then there is  a {\bf compact real form}, 
i.e. a maximal compact subgroup of $G^\C$, defined by a
conjugation $\sigma_c$. Let $\Conj(G^\C)$ be the set of conjugations 
of $G^\C$. We define the following  equivalence
relation in $\Conj(G^\C)$:
$$
\sigma \sim \sigma'  \;\;\mbox{if there is}\;\; \alpha \in
\Int(G^\C)\;\; \mbox{such that}\;\;  \sigma'\,=\,\alpha\sigma
\alpha^{-1}.
$$

We can define a similar relation $\sim$ in the set $\Aut_2(G^\C)$ of automorphisms of
$G^\C$  of order $2$.  Cartan \cite{cartan} shows that there is a bijection  
$$
\mbox{Conj}(G^\C)/\sim \,  \longleftrightarrow\,  \Aut_2(G^\C) /\sim.
$$ 
More concretely, given a compact conjugation $\sigma_c$, in each
class in $\mbox{Conj}(G^\C)/\sim$ one can find a representative $\sigma$ commuting 
with $\sigma_c$ so that $\theta:=\sigma\sigma_c$ is an element of $\Aut_2(G^\C)$, and
similarly if we start with a class in $\Aut_2(G^\C) /\sim$.
The natural map $\Aut_2(G^\C) \lra \Out_2(G^\C)$,  where $\Out_2(G^\C)$ are
the elements of order 2 in $\Out(G^\C)$, is surjective since the extension 
(\ref{outer-extension-group}) splits in this situation (see
\cite{de-siebenthal}), and descends to give  a map 
$\Aut_2(G^\C)/\sim \lra \Out_2(G^\C)$. This  combined with the bijection 
with conjugations  defines a map

$$
c: \Conj(G^\C)/\sim \lra \Out_2(G^\C).
$$ 

Clearly the image of the compact conjugation $\sigma_c$ under this
map is the trivial element in $\Out_2(G^\C)$.

Since $G^\C$ is semisimple we also have the {\bf split real form}, 
defined by a conjugation $\sigma_s$ which can be chosen to commute  with 
the compact conjugation $\sigma_c$.

Consider now a real form $G$ of Hermitian type  of $G^\C$,
defined by a  conjugation $\sigma$. It is well-known that this real form is inner
equivalent to the compact one, that is $c(\sigma)=1$. Starting with  $\sigma$
we can choose  conjugations $\sigma_c$ and $\sigma_s$  commuting with
$\sigma$ and commuting between themselves. This is easy to see if $\sigma_s$
is also inner equivalent to $\sigma_c$, otherwise it requires a little
argument \cite{adams}.
The maps
$\theta:=\sigma\sigma_c$ and $\psi:=\sigma_s\sigma$ are holomorphic involutions
of $G^\C$, i.e.  elements in $\Aut_2(G^\C)$ which commute.
The group  $H=(G^\C)^{\sigma_c}\cap (G^\C)^\sigma$ is a maximal compact subgroup 
of $G$ whose complexification $H^\C$ coincides with $(G^\C)^\theta$.
The following is straightforward.

\begin{proposition}\label{symmetry}
  (1) The involution  $\psi:G^\C\to G^\C$ leaves invariant $G$, $H^\C$ and $H$.

  (2) The differential of $\psi$  preserves the Cartan decomposition,  
  sends $J\in \liez(\lieh)$ to $-J$ and
  exchanges $\liem^+$ and $\liem^-$.

  (3) The involution  of $\pi_1(H)$ induced  by $\psi$ 
  sends $d$ to $-d$ under the isomorphism of the free part of $\pi_1(H)$ 
  with $\Z$. 

  (4) The involution of the symmetric space $G/H$ induced by $\psi$ is an
  antiholomorphic isometry.
\end{proposition}
If now $(E,\varphi^+,\varphi^-)$ is a $G$-Higgs bundle
from  (1) in Proposition \ref{symmetry} we can define the 
$H^\C$-principal bundle
$$
\psi(E):=E\times_\psi H^\C,
$$
and From (2) in  Proposition \ref{symmetry} we have isomorphisms
$$
\psi: E(\liem^\pm)\to \psi(E)(\liem^\mp) 
$$
which  can be used to define the $G$-Higgs bundle  $(\psi(E),\psi(\varphi^-),
\psi(\varphi^+))$ (here we are abusing  notation using $\psi$ also for the 
induced bundle isomorphisms).

The following is a consequence of Proposition \ref{symmetry}.

\begin{proposition} Let $G$ be a real form of Hermitian type of a complex
  semisimple Lie group, and let $\psi$ be defined as above. Then the map
  $$
  (E,\varphi^+,\varphi^-) \mapsto  (\psi(E),\psi(\varphi^-), \psi(\varphi^+))
  $$
  defines an isomorphism between 
  $\cM^\alpha_\tau(G)$ to  $\cM^{-\alpha}_{-\tau}(G)$. In particular 
  it defines an involution of $\cM_0^0(G)$ (simply denoted $\cM_0(G)$).
\end{proposition}

A  simple Lie group $G$ of Hermitian type need not have a complexification as
we are assuming. For instance, any 
non-trivial finite covering $G$ of $\Sp(2n,\R)$ does not sit in any group with 
Lie algebra $\mathfrak{sp}(2n,\C)$, since $\Sp(2n,\C)$ is simply connected and 
we have $\Sp(2n,\R)\subset \Sp(2n,\C)$. However, the symmetric space $M:=G/H$ has always antiholomorphic
involutive automorphisms. These have been classified in
\cite{jaffee1,jaffee2,leung}. From this, one can obtain involutions of the
adjoint group $\Ad(G):=G/Z(G)$ that satisfy the properties in Proposition
\ref{symmetry}. To explain this, let us assume that $G$ is of adjoint type,
i.e. $Z(G)=1$, otherwise we consider $\Ad(G)$.
In this situation, $G$ is the connected component of the identity of $\Isom(M)$,
the group of isometries of $M=G/H$, and consists of course of holomorphic isometries. 
We are now interested in studying conjugations of $M$, that is
anti-holomorphic isometries of $M$. Let us denote by $\Conj(M)$ the set of
all such conjugations. The group $G$ act on $\Conj(M)$ by sending
$\sigma\in \Conj(M)$ to   $g\sigma g^{-1}$ for any $g\in G$. Also 
if $o\in M$ corresponds to the coset $H$, we consider
$\Conj_o(M):=\{\sigma\in\Conj(M)\;:\;\sigma(o)=o\}$. In this case $H$
acts on $\Conj_o(M)$ by conjugation. 
A key result in \cite{jaffee1,jaffee2,leung}
is the following.

\begin{proposition}
  The sets $\Conj(M)/G$ and $\Conj_o(M)/H$ are finite and 
  the natural map $\Conj_o(M)/H\to \Conj(M)/G$ is a bijection.
\end{proposition}

Moreover, \cite{jaffee1,jaffee2,leung} give an explicit classification for the
irreducible symmetric spaces. In all cases the fixed points of any conjugation 
are connected symmetric subspaces $G'/H'\subset G/H$,  whose real dimension 
equals the complex dimension of $G/H$. 

If $G$ is a real form of a complex group $G^\C$  the
involution of $M=G/H$ induced by $\psi$ in Proposition \ref{symmetry} defines
a class in $\Conj_o(M)/H$. It is perhaps plausible that  every class in
$\Conj_o(M)/H$ is defined in a similar way to the one induced by  $\psi$, by replacing the 
split conjugation of $G^\C$   by
another conjugation of $G^\C$ commuting with the compact conjugation of $G^\C$
and the conjugation defining $G$.  

Now, fix $\sigma\in \Conj_o(M)$.  Since $\sigma\in \Isom(M)$ and $G$ is the
connected component of the identity of $\Isom(M)$ we can define an
automorphism $\psi_\sigma$ of $G$ given by the rule $g\mapsto \sigma g
\sigma^{-1}$ for $g\in G$. One can check that this automorphism satisfies 
all the properties in Proposition \ref{symmetry}, besides the fact that,
since $G$ may not have a complexification, $\psi_\sigma$ is only defined 
on $G$, and one has  the following.

\begin{corollary}\label{involutions-adjoint} 
  Let $G$ be a Hermitian Lie group such that $Z(G)=1$. 
  Let $\sigma\in \Conj_o(M)$  and let $\psi_\sigma$ be the involution of $G$
  defined as above. Then the map
  $$
  (E,\varphi^+,\varphi^-) \mapsto  (\psi_\sigma(E),\psi_\sigma(\varphi^-), \psi_\sigma(\varphi^+))
  $$
  defines an isomorphism between
  $\cM^\alpha_\tau(G)$ to  $\cM^{-\alpha}_{-\tau}(G)$.   In particular 
  it defines an involution of $\cM_0(G)$. Moreover the involution
  on the moduli space only depends on the class of $\sigma$ in $\Conj_o(M)/H$.
\end{corollary}

It is clear that $\psi_\sigma$ defines an isometry of the symmetric space
$H^\C/H$ and hence if $h$ is a solution to the Hitchin equations for 
$(E,\varphi)$, i.e. a reduction of structure group of $E$ to $H$ satisfying
(\ref{eq:Hitchin-Kobayashi-h}) then $\psi_\sigma(h)$ is a solution to the Hitchin equations 
for $(\psi_\sigma(E),\psi_\sigma(\varphi))$. As a consequence of this,
the involution of $\cM^\alpha(G)$ is in fact an isometry for the natural
K\"ahler metric defined on $\cM^\alpha(G)$ by solving the Hitchin equations
(see \cite{GGM09}).

If a Hermitian group $G$ has no complexification and is not of adjoint type 
we can consider the adjoint group $\Ad(G)$ and apply to it the previous
results. It is not clear whether the automorphisms of $\Ad(G)$ obtained
from conjugations of $M=G/H$  can be lifted to $G$ to deduce a result similar 
to Proposition \ref{involutions-adjoint}. However, if we consider the map
$\cM^\alpha(G) \to  \cM^\alpha(\Ad(G))$ and now consider the 
involution $\cM^\alpha(\Ad(G))\to \cM^\alpha(\Ad(G))$ defined 
by a $\psi_\sigma$ as above, it is very plausible that this can be lifted in a
compatible way  to an involution of  $\cM^\alpha(G)$. For this, we note that
the class in $H^2(X,Z(G))$ of the $H^\C/Z(G)$-bundle $E'$ associated to an 
$H^\C$-bundle $E$ is trivial, and the same is true for $\psi_\sigma(E')$.

\section{Hermitian groups of tube type and Cayley correspondence}
\label{chap:cayley-correspondence}

In this section we assume that $G$ is a connected, non-compact, real
simple Hermitian Lie group of tube type with finite centre (see
Section \ref{cayley-transform} for definition) with a fixed maximal
compact subgroup $H\subset G$, and $X$ is a compact Riemann surface. We
consider the stability parameter $\alpha$ to be $0$, and, as above, we
refer to $0$-stability of a $G$-Higgs bundle over $X$ simply as
stability (analogously for semistability and polystability).  In this
case the Milnor-Wood inequality for the Toledo invariant of a
$G$-Higgs bundle is given by Theorem \ref{0-theo:ineq-rk}.  We define
a polystable Higgs bundle $(E,\varphi)$ to be \textbf{maximal} if its
Toledo invariant $\tau$ attains one of the bounds of the inequality i.e., $\tau=\pm
r(2g-2)$, where $r=\rk(G/H)$. We denote $\tau_{\max}=\rk(G/H)(2g-2)$.

Let $H^*$ be the non-compact dual of $H$ as defined in Definition \ref{def:non-compact-dual}.  In this section we establish a bijective
correspondence between maximal $G$-Higgs bundles over $X$ and
$K^2$-twisted $H^*$-Higgs bundles over $X$, as defined in Remark
\ref{twisted}, where $K^2$ is the square of the canonical line bundle.

Suppose that $(E,\varphi)$ is a polystable maximal $G$-Higgs bundle, and choose for example $\tau=-r(2g-2)$. By Theorem \ref{0-theo:ineq-rk}, the field $\varphi^+$ has rank $r$ at each point. Let $Z_0^\C\simeq \C^*$ be the connected component of the identity of the center of $H^\C$. There is an exact sequence
\begin{equation}
  1 \longrightarrow Z_0^\C \longrightarrow H^\C \longrightarrow H^\C/Z_0^\C \longrightarrow 1,\label{eq:13}
\end{equation}
so there is an action of $Z_0^\C$-bundles on $H^\C$-bundles that we will
denote by $\otimes$: in this way, if $\kappa$ is a line bundle over $X$, we
can 
define $E\otimes \kappa$ (here we are identifying the line bundle $\kappa$
with its corresponding $\C^\ast$-bundle).

\begin{lemma}
  If $\kappa$ is an $o_J$-root of $K$, where $o_J$ is the order of $e^{2\pi
    J}$, then $\varphi^+$ defines a reduction of the $H^\C$-bundle $E\otimes
  \kappa$ 
  to the group $H'^\C$.
\end{lemma}
\begin{proof}
  Recall that an infinitesimal generator of  $Z_0^\C$ is
  $-io_J J$, so 
  \begin{equation}
    (E\otimes \kappa)(\fm^+)=E(\fm^+)\otimes \kappa^{o_J}=E(\fm^+)\otimes K.\label{eq:12}
  \end{equation}
  Therefore, $\varphi^+$ is a section of $(E\otimes \kappa)(\fm^+)$, and since it has rank $r$ at each point, its stabilizer at each point is isomorphic to $H'^\C$. Therefore $\varphi^+$ defines a reduction of the structure of $E\otimes \kappa$ to $H'^\C$.
\end{proof}

Of course, such $\kappa$ exists only if $o_J$ divides $2g-2$. We will now suppose that $\kappa$ exists and is fixed. Denote by $E'$ the reduction of $E\otimes \kappa$ to $H'^\C$. As we have seen, $\varphi^+\in H^0(X,E'(\fm^+))$, and similarly $\varphi^-\in H^0(X,E'(\fm^-)\otimes K^2)$. 
From Lemma \ref{lemma:cayley-iso}, we have an isomorphism
\begin{equation}
  \label{eq:11}
  \ad \varphi^+ : E'(\fm^-) \longrightarrow E'(\fm'^\C),
\end{equation}
so that we can define a Higgs field 
$$\varphi'=[\varphi^+,\varphi^-]\in H^0(X,E'(\fm'^\C)\otimes K^2).$$ The data $(E',\varphi')$ is a $K^2$-twisted $H^*$-Higgs bundle.

Conversely, from a $K^2$-twisted $H^*$-Higgs bundle $(E',\varphi')$ we can reconstruct $(E,\varphi)$ in the following way. The bundle is $E=E'\otimes \kappa^{-1}$. Observe that for the $H'^\C$-bundle $E'$ we have canonical section $e_\Gamma\in H^0(X,E'(\fm^+))$ corresponding to the element $e_\Gamma\in \fm^+$ fixed by $H'$, which becomes by (\ref{eq:12}) a section $\varphi^+\in H^0(X,E(\fm^+)\otimes K)$. Finally, $\varphi^-$ is reconstructed from (\ref{eq:11}) as $(\ad \varphi^+)^{-1}(\varphi')$. Therefore, $\kappa$ being fixed, we obtain a complete correspondence between maximal $G$-Higgs bundles and $K^2$-twisted $H^*$-Higgs bundles.
We refer to  $(E',\varphi')$ as the \textbf{Cayley partner} of $(E,\varphi)$.
The main result of this section consists in showing that this correspondence preserves stability. 
\begin{theorem}[{\bf Cayley correspondence}]  
  \label{th:cayley-correspondence}
  Let $G$ be a connected non-compact  real simple Hermitian Lie group of tube type
  with finite centre. Let $H$ be a maximal compact subgroup of $G$ and  $H^*$ be 
  the non-compact dual of $H$ in $H^\C$. 
  Let $J$ be the element in $\liez$ (the centre of $\lieh$) defining
  the almost complex structure on $\liem$. 
  If the order of $e^{2\pi J}\in H^\C$ divides $(2g-2)$, then there is an 
  isomorphism of complex algebraic varieties
  \begin{equation}\label{eq:generalized-cayley}
    \cM_{\max} (G) \cong  \cM_{K^2}(H^*)
  \end{equation}
  given by $(E,\varphi)\mapsto(E',\varphi')$ as above.
\end{theorem}

\begin{remark}
  The condition $o_J | (2g-2)$ is always satisfied for a group of
  adjoint type, since in this case $o_J=1$.  Table
  \ref{tab:exponent-tubetype}  shows
  that the $o_J$ divides $(2g-2)$ for the classical and exceptional
  groups. This  may not happen for coverings of these groups, where
  $o_J$ may be bigger.
\end{remark}

The rest of this section is devoted to the proof of Theorem \ref{th:cayley-correspondence}. We begin by the easy direction.
\begin{lemma}
  Let $(E,\varphi)$ be a maximal $G$-Higgs bundle, and let $(E',\varphi')$ be the
  corresponding $K^2$-twisted $H^*$-Higgs bundle. Suppose $(E,\varphi)$ is
  (poly, semi)stable, then $(E',\varphi')$ is (poly, semi)stable.
\end{lemma}
\begin{proof}
  Suppose that we have a reduction $(E'_{P'_s},\sigma')$ of the
  structure group of $E'$ to the parabolic $P'_s$ defined by a $s\in
  i\fh'$, such that $\varphi'$ takes values in
  $E'_{P'_s}(\fm'_s)\otimes K^2$. Since $\fh'\subset \fh$, the element $s$ defines
  a parabolic subgroup $P\subset H^\C$ as well, and using the map
  $H'^\C/P'_s\to H^\C/P_s$ we obtain from $\sigma'$ a reduction $\sigma$ of the
  structure group of $E$ to $P_s$, resulting in a $P_s$-bundle
  $E_P$. Since $s\in i\fh'$ stabilizes $e_\Gamma$, we have $e_\Gamma\in
  \fm^{+,0}_s$, which translates into the fact that $\varphi^+$ takes values in
  $E_{P_s}(\fm^{+,0}_s)\otimes K$. For the same reason, the map $\ad(e_\Gamma)$
  from Lemma \ref{lemma:cayley-iso} sends $\fm^-_s$ to $\fm'^\C_s$, so
  we get that $\varphi^-$ takes values in $E(\fm^-_s)\otimes K)$.

  Therefore, from the reduction $(E'_{P'_s},\sigma')$ we constructed a
  reduction $(E_{P_s},\sigma)$ such that $\varphi$ takes values in
  $E_{P_s}(\fm^\C_s\otimes K)$. By polystability of $(E,\varphi)$, one has
  $\deg(E)(\sigma,s)\geq 0$. Now, since $s\in i\lieh'$ and $J\in i\liem'$ we have
  that $\langle s,iJ\rangle=0$ and hence in the computation of $\deg E(\sigma, s)$ given
  by (\ref{degree-chern-weil}) there is no contribution coming from
  the twisting $E'=E\otimes \kappa$, so that
  \begin{equation}\label{degrees}
    \deg(E)(\sigma,s)=\deg(E')(\sigma',s).
  \end{equation}
  It follows that the (poly, semi)stability of $(E,\varphi)$ implies the
  (poly, semi)stability of $(E',\varphi')$. For polystability, one must just
  check additionally that in the equality case, reduction for the Levi
  subgroup $L_s\subset P_s$ to a $E_{L_s}\subset E$ implies reduction for the Levi
  subgroup $L'_s\subset P'_s$, but it is sufficient to take
  $E_{L'_s}=(E_{L_s}\otimes \kappa)\cap E'$. 
\end{proof}

The other direction is more difficult. We begin by the following
construction: given a $G$-Higgs bundle $(E,\varphi)$, we can associate to it
a $K^2$-twisted $H^\C$-Higgs bundle given by $(E,[\varphi^+,\varphi^-])$, where
$[\varphi^+,\varphi^-]\in H^0(X,E(\lieh^\C)\ot K^2)$ is defined using the Lie
bracket on $\lieh^\C$ combined with the tensor product of $K$ with
itself.  The strategy to prove the Theorem will be to show that if
$(E,\varphi)$ is maximal and $(E',\varphi')$ is its Cayley partner, then the
polystability of $(E',\varphi')$ implies the $\alpha$-polystability of
$(E,[\varphi^+,\varphi^-])$, where $\alpha$ is determined by the topology of $E$. In turn,
to prove this, we will use the correspondence between polystability of
the Higgs bundles involved with solutions to the corresponding Hitchin
equations given by Theorem \ref{theo:hk-twisted-pairs} (see Remark
\ref{theo:-twisted-hk-twisted-pairs} in relation to the
$K^2$-twisting). To complete the proof we will show that the
$\alpha$-polystability of $(E,[\varphi^+,\varphi^-])$, implies the polystability of
$(E,\varphi)$. We prove these various steps separately, so that the Theorem
is a consequence of Lemmas \ref{H*-implies-HC} and
\ref{HC-implies-G}. Again we restrict to the case where $\varphi^+$ is
regular.

\begin{lemma}\label{H*-implies-HC}
  Let $(E,\varphi)$ be a maximal $G$-Higgs bundle, and let $(E',\varphi')$
  be the associated $K^2$-twisted $H^*$-Higgs bundle.  If $(E',\varphi')$ is
  polystable then the $K^2$-twisted $H^\C$-Higgs bundle
  $(E,[\varphi^+,\varphi^-])$ is $-iJ$-polystable, where $J\in \liez(\lieh)$ is
  the element defining the complex structure of $\liem$.
\end{lemma}

\begin{proof}
  By Theorem \ref{theo:hk-twisted-pairs} (see also Remark
  \ref{theo:-twisted-hk-twisted-pairs}) the polystability of $(E',\varphi')$
  implies the existence of a metric on $E'$, that is a smooth section
  $h'$ of $E'(H'^\C/H')$, satisfying the equation
  \begin{equation}
    F_{h'} - [\varphi',\tau_{h'}(\varphi')]\omega= 0.\label{eq:16}
  \end{equation}
  The choice of the fixed K\"ahler form $\omega$ on $X$ is not really important
  here, and to simplify notations we will choose the K\"ahler form of the
  hyperbolic metric, so that $F_K = -i\omega$. We use $\omega$ to define metrics
  on all powers of $K$, especially that on $K^2$ used to define $\tau_h'$
  (see Remark \ref{theo:-twisted-hk-twisted-pairs}).

  Out of the metric $h'$ and the metric on $\kappa$ we obtain a metric $h$ on $E=E'\otimes \kappa^{-1}$; since in (\ref{eq:13}) the infinitesimal generator of $Z_0^\C$ is $o_JJ$, we obtain
  $$ F_h = F_{h'} - \frac1i o_J F_L J = F_{h'} + i F_K J = F_{h'} + \omega J.$$
  Using the Hermite-Einstein equation (\ref{eq:16}) and the identity
  $$
  [\varphi',\tau_{h'}(\varphi')]= [\varphi',\tau_{h}(\varphi')]=
  [[\varphi^+,\varphi^-],\tau_{h}([\varphi^+,\varphi^-])],
  $$
  we therefore obtain 
  $$ F_h- [[\varphi^+,\varphi^-],\tau_{h}([\varphi^+,\varphi^-])] = \omega J. $$
  The lemma follows.
\end{proof}

To relate the stability of a $G$-Higgs bundle $(E,\varphi^+,\varphi^-)$
with the stability of the corresponding $H^\C$-Higgs bundle 
$(E,[\varphi^+,\varphi^-])$ we need the following result from GIT.
\begin{lemma}\label{finite-GIT}
  If $\varphi^+\in \liem^+_{\reg}$, and for some $s\in i\lieh$ we
  have $\varphi^+\in \liem_s^+$, then $\langle-iJ,s\rangle\geq 0$, and if equality holds then $\varphi^+\in \liem^{+,0}_s$.
\end{lemma}
\begin{proof}
  First, define $h=[e_\Gamma,\tau(e_\Gamma)]=h_{\gamma_1}+\cdots +h_{\gamma_r}$. In the tube case, $h$ is an element of the center of $\fh$, and therefore is a multiple of $iJ$. Since for any strongly orthogonal root one has $\gamma(h)=2$ and $\gamma(J)=i$, it follows that
  $$ h=-2iJ. $$
  Observe that the adjoint action of $H$ on $\fm^+$ with its standard flat
  symplectic structure is Hamiltonian, with moment map $\mu$ satisfying
  $i\mu(\varphi^+)=[\varphi^+,\tau(\varphi^+)]$, where $\tau$ here is the
  compact conjugation defining $H$. So in particular, for $\varphi^+=e_\Gamma$, we have 
  \begin{equation}
    i\mu(\varphi^+)=-2iJ.\label{eq:18}
  \end{equation}
  In particular $e_\Gamma$, and therefore the whole $H^\C$-orbit of
  $e_\Gamma$, is polystable for the problem (\ref{eq:18}). Therefore all
  regular elements of $\fm^+$ are polystable. We leave to the reader as an
  exercice in finite dimensional GIT theory that the polystability condition
  is exactly the conclusion of the 
  lemma.
\end{proof}

\begin{lemma}\label{HC-implies-G}
  Let $(E,\varphi)$ be a $G$-Higgs bundle with Toledo invariant $\tau$ such
  that $\varphi^+$ is generically regular (in particular if $(E,\varphi)$ is
  maximal). Let $(E,[\varphi^+,\varphi^-])$ be the associated $H^\C$-Higgs
  bundle. If $(E,[\varphi^+,\varphi^-])$ is $-iJ$-semistable (resp. polystable, stable),
  then $(E,\varphi)$ is semistable (resp. polystable, stable).
\end{lemma}

\begin{proof}
  We must analyse the degree of a reduction $\sigma$ of $E$ to $E_{P_s}$ for a parabolic subgroup $P_s$ such that $\varphi\in H^0(X,E_{P_s}(\fm^\C_s)\otimes K)$. This means that both $\varphi^\pm$ are sections of $E_{P_s}(\fm^\C_s)\otimes K$, so their bracket $[\varphi^+,\varphi^-]$ is a section of $E_{P_s}(\fm^\C_s)\otimes K^2$. So if $(E,[\varphi^+,\varphi^-])$ is $-iJ$-semistable we get
  \begin{equation}
    \det E(\sigma,s) \geq \langle-iJ,s\rangle.\label{eq:17}
  \end{equation}
  Since $\varphi^+$ lies in $E_{P_s}(\fm^+_s)\otimes K$, we can apply Lemma \ref{finite-GIT} to deduce that $\langle-iJ,s\rangle\geq 0$. Therefore $-iJ$-semistability of $(E,[\varphi^+,\varphi^-])$ implies semistability of $(E,\varphi)$. Also the implication for stability follows immediately, since strict inequality in (\ref{eq:17}) implies $\det E(\sigma,s)>0$. 

  For polystability, observe that the equality $\deg E(\sigma,s)=0$ 
  implies $\langle iJ,s\rangle=0$ and the equality in (\ref{eq:17}). Therefore
  $(E,[\varphi^+,\varphi^-])$ reduces to a Levi subgroup $L_s\subset P_s$ and,
  again from Lemma \ref{finite-GIT}, $\varphi^+$ takes values in
  $E_{L_s}(\fm^{+,0}_s)\otimes K)$. From the isomorphism (\ref{eq:11}), the
  condition that  $[\varphi^+,\varphi^-]$ lies in $E_{L_s}(\fm^{-,0}_s)\otimes
  K^2$ implies that $\varphi^-$ lies in $E_{L_s}(\fm^{-,0}_s)\otimes K)$. 
  Therefore $(E,\varphi)$ also reduces to the Levi subgroup $L_s$.
\end{proof}

This lemma finishes the proof of Theorem \ref{th:cayley-correspondence}. See 
Table \ref{tab:tube} for the case of irreducible tube-type Hermitian symmetric spaces
$G/H$.

For the classical groups, Theorem \ref{th:cayley-correspondence} was
proved in \cite{hitchin87,gothen,BGG03,BGG06,BGG15,GGM13}, where it is
sometimes referred to as the 
Cayley  correspondence, inspired by the fact that the symmetric 
space $G/H$ is realized as a tube domain via the Cayley transform 
described in Section \ref{cayley-transform}.

This result is interpreted as a rigidity result for Higgs bundles since the
structure  group of the $K^2$ twisted  $H^*$-Higgs bundles is smaller and
reveals new invariants coming from the group $H^*$. For example, when
$G=\Sp(2n,\R)$,we have  $H^*=\GL(n,\R)$ with $H'=\OO(n)$ as a maximal 
compact subgroup. To a $\Sp(2n,\R)$-Higgs bundle we can thus attach the first 
and second Stiefel-Whitney classes of the principal $\OO(n,\C)$-bundle given by 
the corresponding $GL(n,\R)$-Higgs pair via the Cayley correspondence. 
In general, similar invariants may come from the non-connectedness and 
non-simply connectedness of $H^*$. 

In the previous case by case proofs  many of the geometrical ingredients
were  identified but not explicitly used. Moreover, our result  generalizes
the work  for the classical groups in two ways. First, by 
considering quotients and coverings of the classical groups, even though they 
may not be matrix groups. And second, by including the exceptional case, 
stated as follows.
\begin{theorem}\label{theo:Cayley-exceptional}
  The Cayley correspondence defines an isomorphism of complex algebraic varieties
  $$\cM_{\max}(\E_7^{-25}) \cong \cM_{K^2}(\E_6^{-26}\ltimes \R^*)$$
\end{theorem}

A maximal compact subgroup of $H^*=E_6^{-26}\ltimes \R^*$ is given by
$H'=\mathrm{F}_4\times \Z_2$. Since $H'^\C$ is non-connected, we consider the
short exact sequence $1 \to H'^\C_0 \to H'^\C \to \pi_0(H'^\C)\cong \Z_2 \to
1$ 
and the following homomorphism of its induced long exact sequence in cohomology,
$$H^1(X,H'^\C)\to H^1(X,\pi_0(H'^\C))\cong \Z_2^{2g}.$$
This map associates an invariant in $\Z_2^{2g}$ to any $K^2$-twisted
$H^*$-Higgs  bundle, and hence to any $G$-Higgs bundle. This implies that
$\cM_{K^2}(\E_6^{-26}\ltimes \R^*)$, and hence $\cM_{\max}(\E_7^{-25})$, has at
least $2^{2g}$ connected components. 

The Cayley correspondence can be adapted to $L$-twisted $G$-Higgs bundles as
follows. First,  a version of the inequality of Milnor-Wood for an $L$-twisted 
Higgs bundle gives a different bound  $|\tau|\leq \rk(G/H)\deg L$. 
Let $\cM_{L, \max} (G)$ be the moduli space of $L$-twisted $G$-Higgs

\begin{theorem} \label{th:cayley-correspondence-L-twisted}
  Let $G$ be a connected non-compact  real simple Hermitian Lie group of tube type
  with finite centre. Let $H$ be a maximal compact subgroup of $G$ and  $H^*$ be 
  the non-compact dual of $H$ in $H^\C$. 
  Let $J$ be the 
  element in the centre of the Lie algebra $\lieg$ defining the almost complex 
  structure on $\liem$. If the order of $e^{2\pi J}\in H^\C$ divides $\deg L$, 
  then there is an isomorphism of complex algebraic varieties
  \begin{equation}\label{eq:generalized-cayley-L}
    \cM_{L, \max} (G) \cong  \cM_{L^2}(H^*).
  \end{equation}
\end{theorem}

\section{Hermitian groups of non-tube type and rigidity}
\label{chap:non-tube-domains}

In this section we describe a rigidity phenomenon of maximal $G$-Higgs 
bundles for groups of non-tube type.

\begin{theorem}\label{theo:non-tube-rigidity}
  Let $G$ be a simple Hermitian group of non-tube type and let $H$ be its maximal compact subgroup. Then, there are no stable $G$-Higgs bundles with maximal Toledo invariant. In fact, every polystable maximal $G$-Higgs bundle reduces to a stable $N_G(\lieg_T)_0$-Higgs bundle, where $N_G(\lieg_T)_0$ is the identity component of the normalizer of $\lieg_T$ in $G$, and $\lieg_T$ is the maximal tube subalgebra of $\lieg$.
\end{theorem}


\begin{proof}
  Let $(E,\varphi^+,\varphi^-)$ be a polystable $G$-Higgs bundle. Suppose that the Toledo invariant $\tau$ is maximal. We assume that it is negative, $\tau=-r(2g-2)$ where $r=\rk(G/H)$, without loss of generality. Then, by Theorem \ref{theo:ineq-rk}, $\rk(\varphi^+)=\rk(G/H)=r$. We define $P_r\subset H^\C$, $\sigma$ and $\chi$ as in the proof of Theorem \ref{theo:ineq-rk}. Since the Toledo invariant is maximal, $\deg(E)(\sigma,\chi)=0$, and the polystability condition yields a reduction of $E$ to a Levi subgroup $L\subset P_r$, and the condition $\varphi\in H^0(X,E(\fm^0_{s_\chi})\otimes K)$. By definition of $\chi$ one has $\liel=\ker\ad h$, where $h$ is defined by (\ref{eq:10}), but this is easily seen to be also $\ker(\Ad(c^4)-1)=\tilde \fh_T^\C$. Similarly $\fm^0_{s_\chi}=\fm_T^\C$.

  From Section \ref{cayley-transform}, $\tilde \fg_T=\tilde \fh_T+\fm_T$ is the Lie algebra of the group $N_G(\fg_T)_0$, so the previous observations can be summarized by saying that $(E,\varphi)$ reduces to a $N_G(\fg_T)_0$-Higgs bundle.
  Again from Section \ref{cayley-transform}, note that the maximal compact subgroup of $N_G(\fg_T)_0$ is $N_H(\fh_T)_0$, with Lie algebra $\tilde \fh_T$.
\end{proof}

Note that in the tube case, the argument of Theorem \ref{theo:non-tube-rigidity} does not work since the parabolic subgroup given by Theorem \ref{theo:ineq-rk} is the whole group $H^\C$ and hence, there is no reduction of the structure group.

Theorem \ref{theo:non-tube-rigidity} was proved in \cite{BGG06,BGG15} for
the classical groups. This general approach extends the result to
quotients and coverings and to exceptional groups. For example, every
maximal $\E_6^{-14}$-Higgs bundle is strictly polystable and reduces to
a stable $\Spin_0(2,8)\times_{\Z_4}\U(1)$-Higgs bundle.
From the point
of view of representations similar results were proved in 
\cite{toledo,hernandez,BIW10}.

Continue to consider the maximal case studied in Theorem \ref{theo:non-tube-rigidity}. In the group $G_T$ of the subtube, we have the central subgroup $D'=Z(G_T)\cap H_T$, the adjoint group $G_T^{\Ad}=G_T/D'$, and its maximal compact subgroup $H_T^{\Ad}=H_T/D'$. As already seen in the proof of Theorem \ref{theo:ineq-rk}, we have actually
$$ H_T^{\Ad} = N_H(\fh_T)_0/L', \quad L'=N_H(\fh_T)_0\cap\ker \Ad|_{\fm_T}. $$
Therefore a $N_G(\fg_T)_0$-Higgs bundle gives a $G_T^{\Ad}$-Higgs bundle at
the quotient. From this we have the following.
\begin{theorem}\label{fibration}
  Under the hypotheses of Theorem \ref{theo:non-tube-rigidity}, the maximal representations for $G$ and for the subtube $G_T$ are related by a fibration of complex algebraic varieties,
  $$ \cM_0(L') \longrightarrow \cM_{\max}(G) \longrightarrow \cM_{\max}(G_T^{\Ad}).   $$
\end{theorem}

\begin{proof}
  From Theorem \ref{theo:non-tube-rigidity}, we can replace maximal $G$-Higgs bundles by $N_G(\fg_T)_0$-Higgs bundles. Then
  we have already seen the arrow $\cM_{\max}(G) \longrightarrow \cM_{\max}(G_T^{\Ad})$. Moreover
  we have the exact sequence $1\to L'\to N_G(\fg_T)_0\to G_T^{\Ad}\to1$. But, since there is a direct sum decomposition $\tilde\fg_T=\fg_T\oplus\fl'$, actually
  $$ N_G(\fg_T)_0 = L' \times_{D'} G_T. $$
  Therefore, starting from a maximal $G_T^{\Ad}$-Higgs bundle, we can lift it into a $G_T$-Higgs bundle, the ambiguity being a line bundle $M$ such that $M^{|D'|}=1$. Giving a $L'$-bundle then gives by product a $L'\times G_T^{\Ad}$-Higgs bundle, and passing to the quotient by $D'$ a $N_G(G_T)_0$-Higgs bundle (and the ambiguity on the lift is killed by this quotient). 

  Observe that by maximality, the Toledo invariant of the $G_T^{\Ad}$-Higgs bundle and the Toledo invariant of the $G$-bundle are the same, which implies that the Toledo invariant of the $L'$-bundle vanishes, hence the index `0'.

  Finally observe that the direct sum $\tilde\fg_T=\fg_T\oplus\fl'$ implies that the Hermite-Einstein equation passes to the $L'$ and $G_T^{\Ad}$-bundles, hence polystability is preserved. The fibration is proved.
\end{proof}

We give two applications of the previous theorem. For $p<q$ one has
$$ \cM_0(\U_{q-p}\ltimes \Z_{2p}) \longrightarrow \cM_{\max}(\SU(p,q)) \longrightarrow \cM_{\max}(\PU(p,p)). $$
Since both the basis and the fibres are connected 
\cite{AB83,GO14,BGG03}, this gives a simpler proof 
of the connectedness of $\cM_{\max}(\SU(p,q))$ than in 
\cite{BGG06}.

The second example is the above mentioned case of $\E_6$: we have
$$ \cM_0(\U_1) \longrightarrow \cM_{\max}(\E_6^{-14}) \longrightarrow \cM_{\max}(\PSO_o(2,8)). $$
So the number of connected components of $\cM_{\max}(\E_6^{-14})$ equals that of $\cM_{\max}(\PSO_o(2,8))$. It is known \cite{BGG06} that the space $\cM_{\max}(\SO_0(2,8))$ has $2^{2g+1}$ components, but the space $\cM_{\max}(\PSO_o(2,8))$ may have other components.


\appendix
\label{appendix}


\section{Tables}
\label{chap:tables}

We use the following notation for Table \ref{tab:HSS}:

\begin{itemize}
	\item $\Delta^\pm_{10}$ are the half-spinor representations of the group
	$\Spin(10,\C)$. They are $16$-dimensional.
	\item $M$ and $M^*$ are the irreducible $27$-dimensional representations
	of $\E_6$, which are dual to each other.
	\item $\eta^r$ is the representation $\eta^r: \C^*\to  \C^*$
	given by $z\mapsto z^r$.
\end{itemize}

\newpage

\begin{center}
	
	\begin{table}[htbp]
		\begin{tabular}{|c|c|c|c|}
			\hline\raisebox{-8pt}{}
			$G$ & $H$ &  $H^\C$ & $\liem^\C=\liem^++ \liem^-$ \\
			\hline\hline\raisebox{-8pt}{}
			$\SU(p,q)$ & $\SSS(\U(p)\times \U(q))$ & $\SSS(\GL(p,\C)\times
			\GL(q,\C))$ &
			$\Hom(\C^q,\C^p)+ \Hom(\C^p,\C^q)$ \\
			\hline\raisebox{-8pt}{}
			$\Sp(2n,\R)$ &  $\U(n)$ & $\GL(n,\C)$ &
			$S^2(\C^n) + S^2({\C^n}^*)$ \\
			\hline\raisebox{-8pt}{}
			$\SO^*(2n)$ &  $\U(n)$  & $\GL(n,\C)$ &
			$\Lambda^2(\C^n) + \Lambda^2({\C^n}^*)$ \\
			\hline\raisebox{-8pt}{}
			$\SO_0(2,n)$ & $\SO(2)\times \SO(n)$ & $\SO(2,\C)\times \SO(n,\C)$ &
			$\Hom(\C^n,\C)+ \Hom(\C,\C^n)$ \\
			\hline\raisebox{-8pt}{}
			$\E_6^{-14}$ &  $\Spin(10)\times_{\Z_4} \U(1)$  &
			$\Spin(10,\C)\times_{\Z_4} \C^*$ &
			$\Delta^+_{10}\ot \eta^3 + \Delta^-_{10}\ot \eta^{-3}$ \\
			\hline\raisebox{-8pt}{}
			$\E_7^{-25}$ & $\E_6^{-78}\times_{\Z_3} \U(1)$  & $\E_6\times_{\Z_3}
			\C^*$ &
			$M\ot \eta^2+ M^*\ot \eta^{-2}$ \\

			\hline
			
		\end{tabular}
		\vspace{12pt}
		
		\caption{Irreducible Hermitian symmetric spaces $G/H$}
		\label{tab:HSS}
		
	\end{table}

	\begin{table}[htbp]
		\centering
		\begin{tabular}{|c|c|c|c|c|c|c|}
			\hline\raisebox{-8pt}{}
			$G$ & $H$ & $N$ & $\dim \liem$ & $\ell$ & $o(e^{2\pi J})$ & $q_T$\\
			\hline\hline\raisebox{-8pt}{}
			$\SU(p,q)$ & \footnotesize{$\SSS(\U(p)\times \U(q))$}  & $p+q$ & $2pq$ & \scriptsize{$lcm(p,q)$} & $\frac{p+q}{gcd(p,q)}$ & $\sfrac{1}{2}$ \\
			\hline\raisebox{-8pt}{}
			$\Sp(2n,\R)$ &  $\U(n)$  & $n+1$ &
			$n(n+1)$ & $n$ & $2$ & $\sfrac{1}{2}$ \\
			\hline\raisebox{-8pt}{}
			$\SO^*(2n)$ & $\U(n)$  & $2(n-1)$ & $n(n-1)$ & $n$ & $2$ & $1$\\
			\hline\raisebox{-8pt}{}
			$\SO_0(2,n)$ & $\SO(2)\times \SO(n)$  & $n$ & $2n$ & $1$ & $1$ & $\sfrac{1}{2}$  \\
			\hline\raisebox{-8pt}{}
			$\E_6^{-14}$ &  \footnotesize{$\Spin(10)\times_{\Z_4} \U(1)$} & $12$ & $32$ & $4$ & $3$ & $\sfrac{1}{2}$\\
			\hline\raisebox{-8pt}{}
			$\E_7^{-25}$ & $\E_6^{-78}\times_{\Z_3} \U(1)$    & $18$ & $54$ & $3$ & $2$ & $\sfrac{1}{2}$\\
			\hline
		\end{tabular}
		\vspace{12pt}
		
		\caption{Toledo character data for the classical and exceptional groups}
		\label{tab:exponent-tubetype}\label{tab:exponent-nontubetype} \label{tab:exponent-exceptional}
	\end{table}

	\begin{table}[htbp]
		\centering
		\begin{tabular}{|c|c|c|c|c|c|c|c|}
			\hline\raisebox{-8pt}{}
			$G$ & $H$  & $N$ & $\dim \liem$ & $\ell$ & $o(e^{2\pi J})$ & $q_T$\\
			\hline\raisebox{-8pt}{}
			$\PSU(p,q)$ & \scriptsize{$\mathrm{P}\SSS(\U(p)\times \U(q))$}  & $p+q$ & $2pq$ & \scriptsize{$gcd(p,q)$} & $1$  & $\frac{p+q}{2lcm(p,q)}$ \\
			\hline\raisebox{-8pt}{}
			\footnotesize{$P\SO^*(2n=4m+2)$} & $\U(n)$ & \small{$2(n-1)$} & \small{$n(n-1)$} & $n$ & $1$ & $2$\\
			\hline\raisebox{-8pt}{}
			$\E_6^{-14}/\Z_3$ &  \scriptsize{$\Spin(10)\times_{\Z_4} \U(1)$} & $12$ & $32$ & $4$ & $1$ & $\sfrac{3}{2}$\\
			\hline
		\end{tabular}
		\vspace{12pt}
		\caption{Toledo character data for adjoint groups of non-tube type.}
		\label{tab:exponent-adjoint}
	\end{table}

	\newpage

	\begin{landscape}
		
		\vspace{12pt}
		
		\vspace{12pt}
		
		\vspace{12pt}
		
		\vspace{12pt}

		\renewcommand{\arraystretch}{1.2}
		
		\begin{table}[htbp]
			\begin{tabular}{|c|c|c|c|c|c|c|}
				\hline\raisebox{-8pt}{}
				$G$ & $H$ &  $H^*$ & $H'$ & $\chS=H/H'$ & $\liem'$ &$ \liem'^\C$\\
				\hline\hline\raisebox{0pt}{}
				$\SU(n,n)$ & $\SSS(\U(n)\times \U(n))$ &
				\parbox[t]{110pt}{$\{A \in \GL(n,\C) \st$ \\ \mbox{}\hfill $\det(A)^2 \in \R^+ \}$ \\
				}&
				\parbox[t]{70pt}{$\{A \in \U(n) \st$ \\ \mbox{}\hfill $\det(A)^2 = 1\}$ \\
				}&
				$\U(n)$  & $\Herm(n,\C)$  & $\Mat(n,\C)$  \\
				\hline\raisebox{-10pt}{}
				$\Sp(2n,\R)$ &  $\U(n)$  & $\GL(n,\R)$ & $\OO(n)$ &
				$\U(n)/\OO(n)$  & $\Sym(n,\R)$ &  $\Sym(n,\C)$  \\
				\hline\raisebox{-20pt}{}
				\parbox[t]{40pt}{$\SO^*(2n)$\\ \mbox{}\hfill  $n=2m$} &  $\U(n)$  & $\U^*(n)$ & $\Sp(n)$ &
				$\U(n)/\Sp(n)$  & $\Herm(m,\mathbb{H})$ &  $\Skew(n,\C)$  \\
				\hline\raisebox{-20pt}{}
				$\SO_0(2,n)$ & $\SO(2)\times \SO(n)$ & $\SO_0(1,1)\times \SO(1,n-1)$
				& $\OO(n-1)$ &
				\raisebox{-5pt}{$\dfrac{\U(1)\times S^{n-1}}{\Z_2}$}  & $\R\times \R^{n-1}$ &  $\C\times \C^{n-1}$  \\
				\hline\raisebox{-20pt}{}
				$\E_7^{-25}$ & $\E_6^{-78}\ts_{\Z_3} \U(1)$ & $\E_6^{-26}\ltimes \R^*$ & $\F_4\times \Z_2$ &
				\raisebox{-5pt}{$\dfrac{\E_6^{-78}\cdot \U(1)}{\F_4}$}  & $\Herm(3,\Oc)$ &  $\Herm(3,\Oc)\otimes \C$  \\
				\hline
			\end{tabular}
			\vspace{10pt}
			\caption{Irreducible Hermitian symmetric spaces
				$G/H$ of tube type}
			\label{tab:tube}
		\end{table}
		
		\renewcommand{\arraystretch}{1.2}
		\begin{table}[htbp]
			\begin{tabular}{|c|c|c|c|c|c|c|}
				\hline\raisebox{-8pt}{}
				$G$ & $H$ &  $G_T$ & $L'$ \\
				\hline\hline\raisebox{-8pt}{}
				$\SU(p,q),\, p<q $ & $\SSS(\U(p)\times \U(q))$ &   $\SU(p,p)$ &  $\U(q-p) \ltimes \Z_{2p}$ \\
				\hline\raisebox{-8pt}{}
				$\SO^*(4m+2)$ &  $\U(2m+1)$  & 
				$\SO^*(4m) $ & $\U(1)$  \\
				\hline\raisebox{-8pt}{}
				$\E_6^{-14}$ &  $\Spin(10)\times_{\Z_4} \U(1)$ & 
				$\Spin_0(2,8)$ & $\U(1)$ \\ 
				\hline
			\end{tabular}
			\vspace{10pt}
			\caption{Irreducible Hermitian symmetric
				spaces $G/H$ of non-tube type}
			\label{tab:non-tube}
		\end{table}

	\end{landscape}
	
\end{center}

\clearpage
\newpage

   



\begin{thebibliography}{99}

\bibitem{adams}
  J. Adams, private communication.

\bibitem{AB83}
  M. F. Atiyah and R. Bott, 
  'The {Y}ang-{M}ills equations over {R}iemann surfaces'. {\em Philos. Trans. Roy. Soc. London Ser. A}, {\bf 308} (1983) 523--615.


\bibitem{BGG03}
  S.B. Bradlow, O. Garc{\'{\i}}a-Prada and P.B. Gothen,
  \newblock Surface group representations and {${\rm U}(p,q)$}-{H}iggs bundles.
  \newblock {\em J. Differential Geom.}, {\bf 64} (2003) 111--170.

\bibitem{BGG06}
  S.B. Bradlow, O. Garc{\'{\i}}a-Prada, and P.B. Gothen,
  \newblock Maximal surface group representations in isometry groups of classical
  {H}ermitian symmetric spaces.
  \newblock {\em Geom. Dedicata}, {\bf 122} (2006) 185--213.

\bibitem{BGG13}
  S. B. Bradlow, O.~Garc{\'\i}a-Prada and P. B. Gothen, 
  Deformations of maximal 
  representations in $\Sp(4,\R)$, {\em Q. J. Math.} \textbf{63} (2012) 795--843.

\bibitem{BGG15}
  S. B. Bradlow, O. Garc\'{\i}a-Prada, and P. B. Gothen,
  Higgs bundles for the non-compact dual of the special orthogonal group
  {\em  Geom. Dedicata} {\bf 175} (2015) 1--48.

\bibitem{BGM03}
  S.~B. Bradlow, O.~Garc{\'{\i}}a-Prada, and I.~Mundet~i Riera, Relative
  {H}itchin-{K}obayashi correspondences for principal pairs,
  {\em Quarterly J. Math.},   \textbf{54} (2003) 171--208.

\bibitem{burger-iozzi-labourie-wienhard:2005}
  M.~Burger, A.~Iozzi, F.~Labourie, and A.~Wienhard, Maximal
  representations of surface groups: symplectic {A}nosov structures, {\em Pure
    Appl. Math. Q.} \textbf{1} (2005) 543--590.

\bibitem{BIW10}
  M. Burger, A. Iozzi, and A. Wienhard, 
  \newblock Surface group representations with maximal {T}oledo invariant.
  \newblock {\em Ann. of Math. (2)}  {\bf 172} (2010) 517--566.



\bibitem{cartan}
  \'E. Cartan, Les groupes r\'eels simples, finis et continus,
  \textsl{Ann. \'Ec. Norm. Sup.} \textbf{31} (1914) 263--355.

\bibitem{CO03}
  J.-L. Clerc and B. Orsted, The Gromov norm of the Kaehler class and the
  Maslov index {\em  Asian J. Math.}, {\bf 7} (2003) 269--295.

\bibitem{corlette}
  K.~Corlette, Flat ${G}$-bundles with canonical metrics, \textsl{J. Differential
    Geom.}, \textbf{28} (1988) 361--382.

\bibitem{donaldson}
  S.K. Donaldson, Twisted harmonic maps and the self-duality equations.
  {\em Proc. London Math. Soc.} (3) \textbf{55} (1987) 127--131.

\bibitem{DT87}
  A. Domic and  D. Toledo, The {G}romov norm of the {K}aehler class of symmetric
  domains. {\em Math. Ann.} \textbf{276} (1987) 425--432.

\bibitem{dupont}
  J. L. Dupont, {\em Bounds for characteristic numbers of 
    flat bundles}, Springer LNM, {\bf 763} (1978).

\bibitem{FK94}
  J. Faraut and A. Kor{\'a}nyi,
  \newblock {\em Analysis on symmetric cones}.
  \newblock Oxford Mathematical Monographs. The Clarendon Press Oxford University
  Press, New York, 1994.
  \newblock Oxford Science Publications.

\bibitem{FKKLR00}
  J. Faraut, S. Kaneyuki, A. Kor{\'a}nyi, Q. Lu, and G. Roos,
  \newblock {\em Analysis and geometry on complex homogeneous domains}, volume
  185 of {\em Progress in Mathematics}.
  \newblock Birkh\"auser Boston Inc., Boston, MA, 2000.

\bibitem{GGM09}
  O.~{Garc\'{\i}a-Prada}, P.~B. {Gothen}, and I. Mundet i {Riera},
  \newblock {The Hitchin-Kobayashi correspondence, Higgs pairs and surface group
    representations},  \texttt{arXiv:0909.4487}.

\bibitem{GGM13} 
  O.~Garc{\'\i}a-Prada, P.~B. Gothen, and I.~Mundet i Riera, 
  Higgs bundles and surface group representations in the real
  symplectic group, {\em Journal of Topology}, {\bf 6} (2013) 64--118.

\bibitem{GM04}
  O.~Garc{\'\i}a-Prada and I. Mundet i Riera, Representations of the fundamental
  group of a closed oriented surface in $\Sp(4,\R)$, 
  {\em Topology}, \textbf{43} (2004) 831--855.

\bibitem{GO14}
  O. Garc\'{\i}a-Prada, A. Oliveira, 
  Connectedness of Higgs bundle moduli for complex reductive Lie groups,
  \texttt{arXiv:1408.4778}.

\bibitem{guichard-wienhard:2008}
  O.~Guichard and A.~Wienhard, Convex foliated projective structures and
  the {H}itchin component for $\PSL(4,\R)$, {\em Duke Math. J.}
  \textbf{144} (2008) 381--445.

\bibitem{goldman}
  W.M. Goldman,  Topological components of spaces of representations.
  {\em Invent. Math.} \textbf{93} (1988) 557--607.

\bibitem{gothen}
  P. B. Gothen,  Components of spaces of representations and stable triples,
  {\em Topology} \textbf{40} (2001) 823--850.


\bibitem{helgason}
  S. Helgason,
  \newblock {\em Differential geometry, {L}ie groups, and symmetric spaces},
  volume~34 of {\em Graduate Studies in Mathematics}.
  \newblock American Mathematical Society, Providence, RI, 2001.
  \newblock Corrected reprint of the 1978 original.


\bibitem{hernandez}
  L. Hern\'andez, Maximal representations of surface groups in bounded symmetric
  domains {\em  Trans. Amer. Math. Soc.}, {\bf 324} (1991) 405--420.

\bibitem{hitchin87}
  N.~J. Hitchin, The self-duality equations on a {R}iemann surface,
  {\em Proc. London Math. Soc.}, \textbf{55} (1987) 59--126.

\bibitem{hitchin:duke}
  N.~J. Hitchin,
  Stable bundles and integrable systems, {\em Duke Math. J.} {\bf 54} (1987) 91--114.

\bibitem{hitchin92}
  N.~J. Hitchin,
  \newblock Lie groups and {T}eichm\"uller space.
  \newblock {\em Topology}, {\bf 31}  (1992) 449--473.


\bibitem{HO11}
  T.~{Hartnick} and A.~{Ott}, Milnor-Wood type inequalities for Higgs bundles, \texttt{arxiv:1105.4323}.

\bibitem{knapp}
  A. W. Knapp, {\em Lie Groups beyond an Introduction}, first ed.,
  Progress in Mathematics, vol 140, Birkh\"auser Boston Inc., Boston,
  MA, 1996.

\bibitem{jaffee1}
  H.A. Jaffee, Real forms of Hermitian symmetric spaces, 
  {\em Bulletin of the AMS} {\bf 81} (1974) 456--458.

\bibitem{jaffee2}
  H.A. Jaffee, Anti-holomorphic automorphisms of the exceptional symmetric domains,
  {\em J. Differential Geometry} {\bf 13} (1978) 79--86.

\bibitem{kobayashi}
  S. Kobayashi,
  \newblock {\em Differential geometry of complex vector bundles}, volume~15 of
  {\em Publications of the Mathematical Society of Japan}.
  \newblock Princeton University Press, Princeton, NJ, 1987.
  \newblock Kan{\^o} Memorial Lectures, 5.



\bibitem{KV79}
  A. Kor{\'a}nyi and S. V{\'a}gi,
  \newblock Rational inner functions on bounded symmetric domains.
  \newblock {\em Trans. Amer. Math. Soc.}, \textbf{254} (1979) 179--193.

\bibitem{KW65}
  A. Kor{\'a}nyi and J.A. Wolf,
  \newblock Realization of hermitian symmetric spaces as generalized half-planes.
  \newblock {\em Ann. of Math. (2)}, \textbf{81} (1965) 265--288.

\bibitem{labourie}
  F. Labourie, Anosov flows, surface groups and curves in projective space,
  {\em Invent. Math.} \textbf{165} (2006) 51--114.

\bibitem{leung}
  D.S.P. Leung, Reflective submanifolds.IV. Classification of real forms of
  Hermitian symmetric spaces, {\em J. Differential Geometry} {\bf 14} (1979) 179--185.

\bibitem{milnor}
  J.~W. Milnor, On the existence of a connection with curvature zero,
  {\em Commm. Math. Helv.} {\bf 32} (1958) 215--223.

\bibitem{nitsure}
  N. Nitsure, Moduli space of semistable pairs on a
  curve, {\em Proc. London Math. Soc.} {\bf 62} (1991) 275--300.


\bibitem{ramanathan96} 
  A.~Ramanathan,
  {Moduli for principal
    bundles over algebraic curves: I and II}, {\em Proc. Indian
    Acad. Sci. Math. Sci.}  \textbf{106} (1996) 301--328 and 421--449.

\bibitem{Rub12}
  R. Rubio, \emph{Higgs bundles and Hermitian symmetric spaces}, PhD thesis, Universidad Auton\'oma de Madrid, 2012.

\bibitem{schmitt05}
  A.W.H. Schmitt,
  \newblock Moduli for decorated tuples of sheaves and representation spaces for
  quivers.
  \newblock {\em Proc. Indian Acad. Sci. Math. Sci.}, \textbf{115} (2005) 15--49.

\bibitem{schmitt08}
  A. H.~W. Schmitt,
  \newblock {\em Geometric invariant theory and decorated principal bundles}.
  \newblock Zurich Lectures in Advanced Mathematics. European Mathematical
  Society (EMS), Z\"urich, 2008.

\bibitem{de-siebenthal}
  J. de Siebenthal,
  Sur les groupes de Lie compact non-connexes,
  \textsl{Commentari Math. Helv.} \textbf{31} (1956) 41--89.

\bibitem{simpson94}
  C. T. Simpson, {Moduli of representations of the fundamental group of a smooth
    projective variety {I}}, {\em Publ. Math., Inst. Hautes \'Etud. Sci.}
  \textbf{79}   (1994) 47--129.

\bibitem{simpson95}
  C. T. Simpson,
  {Moduli of representations of the fundamental group of a smooth
    projective variety {II}}, {\em Publ. Math., Inst. Hautes \'Etud. Sci.}
  \textbf{80}   (1995) 5--79.


\bibitem{toledo}
  D. Toledo,
  \newblock Representations of surface groups in complex hyperbolic space.
  \newblock {\em J. Differential Geom.}, \textbf{29} (1989) 125--133.

\bibitem{wood}
  J. W. Wood, Bundles with totally disconnected structure group,
  {\em  Comment. Math. Helv.} {\bf 46} (1971) 257--273.

\end{thebibliography}
\end{document}